\documentclass[a4paper]{article}
\usepackage{amsfonts,dsfont,amsmath,amsthm,enumerate,color,fullpage}
\usepackage{subfigure}

\theoremstyle{plain}
\newtheorem{theorem}{Theorem}[section]
\newtheorem{corollary}[theorem]{Corollary}
\newtheorem{lemma}[theorem]{Lemma}
\newtheorem{proposition}[theorem]{Proposition}

\newtheorem{assumption}[theorem]{Assumption}
\newtheorem{example}[theorem]{Example}

\numberwithin{equation}{section}

\theoremstyle{remark}
\newtheorem{remark}{Remark}

\theoremstyle{definition}
\newtheorem{definition}{Definition}

\newcommand\stwoscale{\overset{_2}{\rightarrow}}

\newcommand{\ep}{\varepsilon}

\newcommand{\beq}{\begin{equation}}
\newcommand{\eeq}{\end{equation}}
\newcommand{\beqs}{\begin{equation*}}
\newcommand{\eeqs}{\end{equation*}}
\newcommand{\bit}{\begin{itemize}}
\newcommand{\eit}{\end{itemize}}
\newcommand{\ben}{\begin{enumerate}}
\newcommand{\een}{\end{enumerate}}
\newcommand{\bal}{\begin{align}}
\newcommand{\eal}{\end{align}}
\newcommand{\bals}{\begin{align*}}
\newcommand{\eals}{\end{align*}}
\newcommand{\bse}{\begin{subequations}}
\newcommand{\ese}{\end{subequations}}
\newcommand{\bpr}{\begin{proposition}}
\newcommand{\epr}{\end{proposition}}
\newcommand{\bre}{\begin{remark}}
\newcommand{\ere}{\end{remark}}
\newcommand{\bpf}{\begin{proof}}
\newcommand{\epf}{\end{proof}}
\newcommand{\ble}{\begin{lemma}}
\newcommand{\ele}{\end{lemma}}
\newcommand{\bco}{\begin{corollary}}
\newcommand{\eco}{\end{corollary}}
\newcommand{\bex}{\begin{example}}
\newcommand{\eex}{\end{example}}
\newcommand{\bth}{\begin{theorem}}
\newcommand{\enth}{\end{theorem}}

\newcommand{\eps}{\varepsilon}

\def\XXint#1#2#3{{\setbox0=\hbox{$#1{#2#3}{\int}$}
     \vcenter{\hbox{$#2#3$}}\kern-.5\wd0}}

\begin{document}



\title{\itshape Two-scale homogenization for a general class of high contrast PDE systems with periodic coefficients}

\author{\vspace{6pt}I.V. Kamotski\footnote{Email: i.kamotski@ucl.ac.uk}\ 
 and V.P. Smyshlyaev\footnote{ Email: v.smyshlyaev@ucl.ac.uk }
\\\vspace{0pt}  {\em{Department of Mathematics, University College London,}}
\\\vspace{0pt}  {\em{ Gower Street, London, WC1E 6BT, UK}}
 }
\date{}
\maketitle

\begin{abstract}
For two-scale homogenization of a general class of asymptotically degenerating 
strongly elliptic symmetric PDE systems with a critically scaled high contrast in periodic 
coefficients of a small period $\ep$, we derive a two-scale limit resolvent problem under a single generic  
decomposition assumption for the `stiff' part. We show that this key assumption does hold for a large number 
of examples with a high contrast, both studied before and some recent ones, including those in linear elasticity and electromagnetism. Following ideas of V.V. Zhikov, under very mild restrictions on the regularity of the domain $\Omega$ we prove that the limit resolvent problem is well-posed and turns out to be a pseudo-resolvent problem for a well-defined non-negative self-adjoint two-scale limit operator. A key novel technical ingredient here 
is a proof that the linear span of product test functions in the functional spaces corresponding to the degeneracies is dense in associated two-scale energy space for a general coupling between the scales. 
As a result, we establish (both weak and strong) two-scale resolvent convergence, as well as some of its further implications for the spectral convergence and 
for convergence of parabolic and hyperbolic semigroups and of associated time-dependent initial boundary 
value problems.



\end{abstract}
\section{Introduction}

This work is dedicated to memory of Professor V. V. Zhikov. One of V.V. Zhikov's many important contributions was 
development of a powerful operator theoretic and spectral approach for two-scale convergence and its application to 
double-porosity type models, see e.g. \cite{Zhikov2000,Zhikov2005,ZhP07,ZhP13}. The latter models are 
examples of high-contrast homogenization problems where the solutions behave non-classically in the sense that 
they retain  a two-scale pattern in their limit asymptotic behaviour. This is a source of a number of interesting 
physical effects, and a reason for applying and developing non-classical mathematical tools for analysis 
of such problems. 

In this context, simplest double-porosity type models are divergence form partial differential equations 
(PDEs) with 
$\ep$-periodic coefficients and with a high contrast between the coefficients for the `stiff' and `soft' phases. 
It has been observed starting from at least \cite{FenKhr80} that for a critically scaled contrast $\delta$, 
$\delta=O(\ep^2)$, the solutions' asymptotic behavior may display various interesting effects.  It was 
shown in \cite{ADH90} that certain macroscopic porous media flow models can be derived as two-scale homogenized limits of two-component Darcy flows with $O(\ep^2)$-contrasting properties. Various classes of $O(\ep^2)$ 
high-contrast homogenization problems were studied since. Without attempting here a comprehensive review, 
such problems and related physical effects and mathematical issues have since been intensively studied among others in \cite{AurBonn85,Khrus,panasenko,allaire, Sandr99,Zhikov2000,Zhikov2005,BF04,CSZ06,BKS08, KohnShip08, AGMR08, 
Cherd09, VPS09, bellieud,cooperPhD,ZhP13,ChenLip13,Cooper14,CKS14,CC15,BBF17}. 

As probably first observed by Khruslov, see e.g. \cite{Khrus}, homogenization of certain parabolic problems 
with high contrast leads to weakly coupled systems with memory, i.e. with a non-locality in time. Zhikov 
\cite{Zhikov2000,Zhikov2005} analyzed such models using tools of two-scale operator convergence, which 
not only confirmed the time non-locality but for particular cases also established the spectral convergence and 
existence of frequency band gaps. Such a non-local behavior as well as the asymptotic description of the 
band gaps appear closely related to the so-called `negative' materials where, for certain ranges of 
frequencies, some materials with an $\ep^2$-contrast behave as if they had certain  macroscopic properties negative-valued, which 
was first observed  probably by Auriault and Bonnet \cite{AurBonn85} by formal asymptotics, and 
followed by mathematical analysis of related diffraction problems in \cite{BF04}, see e.g. \cite{BBF17}. 
As was 
shown in \cite{CSZ06}, appropriately modified models may display a spatial non-locality 
by introducing not only high contrast but also a high anisotropy which may be viewed as a particular case 
of a `partial degeneracy' where some components of the `stiffness' matrix remain of order one while others 
asymptotically degenerate and are of order $\ep^2$. Spatial non-locality appears a generic feature 
for certain classes of high-contrast media, see e.g. \cite{CamEddSepp03}, and appears also a generic property  under `ensemble averaging' 
of composite materials \cite{MBW06}. 
It was further shown in \cite{VPS09} by formal 
asymptotics that in linear elastic context $O(\ep^2)$ partial degeneracies may be capable of leading to some 
sort of combined spatio-temporal non-locality. A particular such model of partial degeneracy where 
isolated elastic inclusions have order $\ep^2$ shear modulus but order one bulk modulus was the rigorously 
analyzed in \cite{Cooper14}. As shown in \cite{CKS14}, for a photonic crystal fiber type waveguide 
structure with an `almost critical' wave propagation constant along the fibers, the problem can be reduced 
to another `partially degenerating' one. Analysis of fully three-dimensional Maxwell's systems with 
high contrast in electric permittivity, cf \cite{CC15, BBF17}, appears also to display a kind of partial 
degeneracy due to intrinsic degeneracies of the Maxwell's system. 

The above background, and in particular the increasing list of examples of high contrast models with 
partial degeneracies and of associated additional effects, motivates an attempt to analyze such problems 
mathematically 
in a general setting, as we undertake in the present work. With this aim, we consider a general class of 
strongly elliptic symmetric PDE systems with $\ep$-periodic coefficients having a most general order-$\ep^2$
degeneration in their coefficients, i.e. without necessarily any separate stiff and soft phases at all, see \eqref{aeps}. 

As Zhikov has demonstrated, see e.g. \cite{Zhikov2000}, and as then 
further clarified by Zhikov and Pastukhova \cite{ZhP07,Past05}, analysis of convergence for associated 
resolvent problem is fundamental for operator and spectral convergences as well as for convergences of 
associated semigroups and of related time dependent evolution problems, both parabolic and hyperbolic. 
We therefore thoroughly analyze the associated general resolvent problem \eqref{pde}. Under generic 
conditions \eqref{rhobd}--\eqref{coerc1} for symmetry, boundedness and strong ellipticity, we employ 
the tools of two-scale convergence \cite{nguetseng,allaire,Zhikov2000} to pass to the (two-scale) 
limit in \eqref{pde}. To achieve this, we introduce a single generic decomposition
assumption \eqref{keyass} for the `stiff' part $a^{(1)}(y)$ and show that this assumption does hold for a large number
of examples involving an $\ep^2$-contrast, both studied before and some recent ones. 
A curious observation 
is that 
for the particular case of constant $a^{(1)}$,   
\eqref{keyass} appears to be equivalent to a `constant rank' assumption for $a^{(1)}$, with a similar 
assumption implying a similar key 
decomposition property in the 
$ {\cal A}$-quasiconvexity theory of Fonseca and M\"{u}ller \cite{FM} ensuring lower semi-continuity of a wide class of variational functionals 
subject to differential constraint $ {\cal A}v:=a^{(1)}\nabla v=0$.

We then show that 
the above key decomposition assumption implies a generalization of Weyl's decomposition (Theorem \ref{weyldec}), and 
allows to develop some form of generalized two-scale coupled corrector problem and of 
associated relation between the two-scale limit fields and fluxes, see \eqref{fluxcorr}. 
This in turn allows to pass to the limit in the variational formulation \eqref{weak} of \eqref{pde} 
for appropriate product
test functions in the functional spaces corresponding to the degeneracies, see \eqref{limweak}. 
This determines a limit two-scale operator form, and one of the main novel technical ingredients of this 
work is a proof that, under very mild restrictions on the
regularity of the domain $\Omega$ (see Remark \ref{remepigr}), 
 the linear span of the product
test functions  is dense in associated
two-scale energy space $U$ for a general coupling between the scales, Theorem \ref{cinfdense}. 

The above allows to pass to the limit in \eqref{weak}, which leads to a well-posed two scale problem, 
Theorem \ref{maintheor}. This has numerous further implications: a well-defined  limit operator 
$A_0$ as a two-scale non-negative self-adjoint operator in a Hilbert space $H_0\subset L^2(\Omega\times Q)$ 
where $Q$ is the unit cell, and ensued interpretation of Theorem \ref{maintheor} in terms of a weak 
two-scale (pseudo-)resolvent convergence, Corollary \ref{w2src2}, cf \cite{Zhikov2000,ZhP07,Past05}. 
This implies associated strong two-scale resolvent convergence (Theorem \ref{s2src}), 
which has in turn subsequent implications
for the spectral convergence (Corollaries \ref{h1spec} and \ref{specproj}) and for convergence of parabolic and hyperbolic semigroups and
of associated time-dependent initial boundary value problems (Theorems \ref{parsem} and 
\ref{hyperbcauchy}).


\section{Formulation}

\subsection{Resolvent problem for high contrast PDE systems}

We consider the following general resolvent-type boundary value problem in domain $\Omega \subset {\mathbb R}^d$, $d\geq 1$
\begin{equation}
-\,\,\mbox{div} \Bigl(\, a^\varepsilon(x) {\nabla} u\, \Bigr)\,+\,\lambda\, \rho^\varepsilon(x)\, u\,=\,
\rho^\ep(x) f^\varepsilon(x).
\label{pde}
\end{equation}
The domain $\Omega$ can a priori be any open set in ${\mathbb R}^d$, both bounded or unbounded (in
particular $\Omega={\mathbb R}^d$). Here $u \in
\left(H^1_0(\Omega)\right)^n$, $n\geq 1$, is the sought
(possibly vector-valued) function, $\lambda>0$ is a real
positive (spectral) parameter, $0<\varepsilon<1$ is a small
parameter. The right hand side $f^\varepsilon\in\left(L^2(\Omega)\right)^n$ is generally assumed uniformly  bounded in
$\left(L^2(\Omega)\right)^n$ with respect to $\varepsilon$. 

The density $\rho^\varepsilon(x)$ is assumed to be in general a bounded and uniformly positive $\varepsilon$-periodic symmetric matrix:
\begin{eqnarray}
 \rho^\varepsilon(x)\,=\,\rho\left(\frac{x}{\varepsilon}\right), \ \ \rho(y)\in \left(L_{\#}^\infty(Q)\right)^{n\times n}, \ \ 
\rho_{ij}(y)=\rho_{ji}(y), \ \nonumber\\ 
\rho_{ij}(y)\xi_i\xi_j\geq \nu\vert\xi\vert^2, \ \  \nu>0, \,\forall \xi\in {\mathbb R}^n, \ 
 \mbox{for a.e. } y \in Q,
\label{rhobd}
\end{eqnarray}
where the unit cube $Q=[0,1)^d$ is the periodicity cell of the `fast variable' $y\in {\mathbb R}^d$. 
In (\ref{rhobd}) and henceforth summation is  implied with respect to repeated indices, 
$L^\infty_\#(Q)$ denotes functions from $L^\infty\left({\mathbb R}^d\right)$ which are $Q$-periodic.

The rapidly oscillating 
tensor $a^\varepsilon(x)$ is allowed to degenerate as $\varepsilon \to 0$, as follows:
\begin{equation}
a^\varepsilon(x)\,=\,a^{(1)}\left(\frac{x}{\varepsilon}\right)\,+\,
\varepsilon^2\,a^{(0)}\left(\frac{x}{\varepsilon}\right),
\label{aeps}
\end{equation}
where
\begin{equation}
a^{(l)}(y) \in \left(L^\infty_\#(Q)\right)^{n\times d \times n \times d}, \ \ \ l=0,1, 
\label{alinfty}
\end{equation}
are symmetric:
\begin{eqnarray}
 a^{(l)}(y)\,=\,\left(a^{(l)}_{ijpq}(y)\right), \,\,\, 1\leq i,p \leq n, \,\,\, 1\leq j,q \leq d, \ \ 
\nonumber\\ 
 a^{(l)}_{ijpq}(y)\,=\, a^{(l)}_{pqij}(y), \,\, \forall i,j,p,q, \mbox{ for a.e. }y\in Q.
\label{symm}
\end{eqnarray}
The tensor $a^{(1)}$ is further assumed to be non-negative, i.e.
\begin{equation}
a^{(1)}_{ijpq}(y)\zeta_{ij}\zeta_{pq} \,\geq\,0, \ \ \ 
\forall \zeta \in {\mathbb R}^{n\times d}, \mbox{ for a.e. }y\in Q.
\label{nonneg}
\end{equation}

The tensor $a^{(0)}$ is in turn assumed to be such that $a^{(0)}(y)+a^{(1)}(y)$ is
strongly uniformly elliptic, in the sense that 
\begin{equation}
\int_{{\mathbb R}^{d}} \biggl[ a^{(1)}\left(y\right) {\nabla} u(y)\cdot \nabla u(y)+
 a^{(0)}\left(y\right) {\nabla} u(y)\cdot \nabla u(y) \biggr]dy\geq \nu \| \nabla u\|^2_{(L^2({\mathbb R}^{d}))^{n\times d}},  
\ \forall u\in \left(H^1({\mathbb R}^{d})\right)^n, 
\label{coerc1}
\end{equation}
with some 
constant $\nu>0$ independent of  $u$. We remark that while \eqref{coerc1} seems 
the most general condition of strong ellipticity for $a^{(0)}(y)+a^{(1)}(y)$, the condition \eqref{nonneg} of non-negativity for $a^{(1)}(y)$ may be slightly restrictive: for example, for constant $a^{(1)}$ the condition 
ensuring (in the absence of $a^{(0)}$) \eqref{coerc1} would be 
$a^{(1)}_{ijpq}\xi_i\eta_j\xi_p\eta_q\geq \nu |\xi|^2|\eta|^2$, $\forall \xi\in{\mathbb R}^n,
\eta\in{\mathbb R}^d$, which does not generally imply \eqref{nonneg}. 
However as we illustrate in Section \ref{keyassexamples}, condition \eqref{nonneg} which is essential for the present method, 
appears to hold for numerous systems from physics, notably from linear elasticity and electromagnetism. 

 For a fixed $\varepsilon >0$, for any $\lambda>0$ 
the boundary value problem (\ref{pde}) admits an equivalent weak formulation as follows: 
find $u\in \left(H_0^1(\Omega)\right)^n$ such that 
\begin{eqnarray}
\int_\Omega \biggl[ a^{(1)}\left(\frac{x}{\varepsilon}\right)  {\nabla} u\cdot
\nabla\phi(x)&+&\varepsilon^2 a^{(0)}\left(\frac{x}{\varepsilon}\right)  {\nabla} u\cdot
\nabla\phi(x)\,+ \ \ \ \ \ \ 
\nonumber\\ 
\ \ \ \ \ \ \ \ \ \ \lambda \rho\left(\frac{x}{\varepsilon}\right) u\cdot \phi
(x) \biggr]dx\,&=& \int_\Omega \rho\left(\frac{x}{\ep}\right)f^\varepsilon(x)\cdot
\phi(x)\,dx,  \ \ \ \forall \phi\in\,\left(
H^1_0(\Omega)\right)^n. \label{weak}
\end{eqnarray}

For any fixed positive  $\varepsilon$ and $\lambda$, the conditions (\ref{rhobd})--(\ref{coerc1}) 
immediately ensure applicability of standard theory, with 
Lax-Milgram lemma, see e.g. \cite{evans}, guaranteeing the existence of a 
unique solution in $\left( H_0^1(\Omega)\right)^n$, denoted $u^\varepsilon$.

Problem \eqref{weak} can be regarded, in a standard way, as a resolvent problem for a non-negative self-adjoint operator $A_\ep$ as follows. Consider Hilbert space $H_\ep=\left(L^2(\Omega)\right)^n$ with inner product 
$(u,v)_{H_\ep}:=\int_\Omega u(x)\cdot\rho^\ep(x)v(x)dx$. Then the sesquilinear form  defined on 
$V=\left(H_0^1(\Omega)\right)^n$ by the left hand side of \eqref{weak} with $\lambda=1$, due to 
\eqref{rhobd}--\eqref{coerc1}, defines an equivalent inner product in $\left(H_0^1(\Omega)\right)^n$ and is 
hence densely defined in $H_\ep$, non-negative and closed. It therefore defines a non-negative self-adjoint 
operator $A_\ep$ with a domain $D(A_\ep)$ dense in $V$. This recasts \eqref{pde}, equivalently \eqref{weak}, 
as a resolvent problem in $H_\ep$: 
\beq
\left(\,A_\ep\,+\,\lambda I\right)u_\ep\,=\,f_\ep \ \  \Longleftrightarrow \ \ 
u_\ep\,=\,\left(\,A_\ep\,+\,\lambda I\right)^{-1}f_\ep, 
\label{resolveps}
\eeq
with $I$ denoting the identity operator. 

The interest is in establishing a version of `resolvent convergence', i.e. for any $\lambda>0$ in passing to an appropriate limit, as $\varepsilon \to 0$, for $u^\varepsilon$ whenever 
$f^\ep$ converges to some $f_0$ (in an appropriate sense). 

\subsection{Basic definitions and properties of two-scale convergence} 

For passing to the limit in \eqref{weak} we employ traditional recipes of two-scale convergence, 
see e.g. \cite{nguetseng,allaire, Zhikov2000}. We list below some basic definitions and properties of 
the two-scale convergence, in a form closest to Zhikov see e.g. \cite{Zhikov2000,ZhP07} as 
adapted to our context.  

We will denote by $C_0^\infty(\Omega)$ and $C_\#^\infty(Q)$ the linear spaces of all (test) functions which are infinitely differentiable, and respectively compactly supported in domain $\Omega$ and $Q$-periodic in 
${\mathbb R}^d$. 
For an arbitrary open domain $\Omega$ in ${\mathbb R}^d$, a bounded sequence $\{u^\ep(x)\}$ in 
$L^2(\Omega)$
is said to {\it weakly two-scale converge} to a function $u(x,y)$ in $L^2(\Omega\times Q)$, 
denoted $u^\varepsilon(x)\stackrel{2}\rightharpoonup u(x,y) $, if 
\beq
\lim_{\ep\to 0}\int_\Omega u^\ep(x)\phi(x)b\left(\frac{x}{\ep}\right)dx=
\int_\Omega\int_Q u(x,y)\phi(x)b(y)dxdy,  \ \ \forall\phi(x)\in C_0^\infty(\Omega), b(y)\in C^\infty_\#(Q). 
\label{w2scdef}
\eeq
The weak two-scale limit is unique since the linear span of $\phi(x)b(y)$, 
$\phi(x)\in C_0^\infty(\Omega), b(y)\in C^\infty_\#(Q)$, is dense in $L^2(\Omega\times Q)$. 
The sequence is said to {\it strongly two-scale converge} to $u(x,y)\in L^2(\Omega\times Q)$, 
denoted $u^\varepsilon(x)\stackrel{2}\rightarrow u(x,y)$, if 
\beq
\lim_{\ep\to 0}\int_\Omega u^\ep(x)v^\ep(x)dx\,=\,
\int_\Omega\int_Q u(x,y)v(x,y)dx\,dy \ \ \ \mbox{whenever } v^\varepsilon(x)\stackrel{2}\rightharpoonup v(x,y). 
\label{s2scdef}
\eeq
We will recall a key compactness property of weak two-scale convergence: every bounded sequence $u^\ep$ in 
$L^2(\Omega)$ has a subsequence which weakly two-scale converges to some $u(x,y)\in L^2(\Omega\times Q)$. 
Another simple property of the two-scale convergence on which we will rely is that if 
$u^\varepsilon(x)\stackrel{2}\rightharpoonup u(x,y) $ 
(respectively $u^\varepsilon(x)\stackrel{2}\rightarrow u(x,y) $) 
and $b(y)\in L_\#^\infty(Q)$ then 
$b(x/\ep)u^\varepsilon(x)\stackrel{2}\rightharpoonup b(y)u(x,y)$ (resp 
$b(x/\ep)u^\varepsilon(x)\stackrel{2}\rightarrow b(y)u(x,y)$). 
For  $u^\varepsilon(x)\stackrel{2}\rightharpoonup u(x,y)$, 
\beq
\liminf_{\ep\to 0}\|u^\ep(x)\|_{L^2(\Omega)}\geq\|u(x,y)\|_{L^2(\Omega\times Q)}, 
\label{liminf}
\eeq
and strong two-scale convergence is equivalent to weak two-scale convergence 
in conjunction with the convergence of the 
norms: 
\beq 
u^\varepsilon(x)\stackrel{2}\rightarrow u(x,y) \, \, \Longleftrightarrow \, \,
u^\varepsilon(x)\stackrel{2}\rightharpoonup u(x,y) \ \mbox{and } 
 \lim_{\ep\to 0}\|u^\ep(x)\|_{L^2(\Omega)}=\|u(x,y)\|_{L^2(\Omega\times Q)}. 
\label{weaknormstrong}
\eeq
The strong two-scale convergence implies, assuming sufficient regularity of $u_0(x,y)$ e.g. 
$u_0\in L^2(\Omega; C_\#(Q))$ (\cite{allaire} Theorem 1.8), that 
$\left\| u^\ep(x)-u_0\left(x,x/\ep\right)\right\|_{L^2(\Omega)} \to 0$ as ${\ep\to 0}$. 
For further properties of two-scale convergence, see e.g. \cite{nguetseng,allaire, Zhikov2000,ZhP07}.

The above is immediately extended to vector or tensor-valued functions, in the component-wise sense. For example, 
we regard a matrix-valued $\xi^\ep(x)=\left\{\xi^\ep_{ij}(x)\right\}, 1\leq i\leq n$, $1\leq j\leq d$ to 
weakly (strongly) two-scale converge to $\xi^0(x,y)=\left\{\xi^0_{ij}(x,y)\right\}$ if simply the above 
definitions hold for every $i$ and $j$. Remark also that, due to \eqref{weak} associated with the resolvent problem 
\eqref{resolveps} in Hilbert space $H_\ep$ generally with (matrix) weights $\rho^\ep$, it could be natural to 
operate with two-scale convergence with respect to a (matrix) measure $\mu_\ep=\rho^\ep dx$, cf. e.g. 
\cite{Zhikov2000}. This however appears not necessary for the purposes of the present work, due to the imposed 
in \eqref{rhobd} uniform positivity and boundedness of $\rho^\ep$: as a result, the two notions of two-scale convergence are equivalent.


\section{A priori estimates and functional spaces for two-scale limits.}

\subsection{A priori estimates} 

In this subsection, for a fixed $\lambda>0$, we derive in a standard way a priori estimates for the solution $u^\varepsilon$ of \eqref{weak}. 
Henceforth $C$ denotes a positive constant, independent of $\ep$ and $f^\ep$, 
 whose precise value is insignificant and can change from line to line;  
$\|\cdot\|_2$ denotes appropriate $L^2$-norm. 

\begin{lemma}
For $0<\ep<1/2$, the following a priori estimates hold:
\begin{eqnarray}
\| u^\varepsilon \|_2 &\leq & C \| f^\varepsilon \|_2, \label{apest1} \\
\| \varepsilon \nabla u^\varepsilon \|_2 &\leq & C \| f^\varepsilon \|_2, \label{apest2} \\
\left\| \bigl( a^{(1)}(x/\varepsilon)\bigr)^{1/2} \nabla u^\varepsilon \right\|_2 &\leq & C 
\| f^\varepsilon \|_2, \label{apest3} 
\end{eqnarray}
with a constant $C$ independent of $\varepsilon$ and $f^\ep$. 
\end{lemma} 
\begin{remark}
Notice that in \eqref{apest3} $\left(a^{(1)}(y)\right)^{1/2}$ is well-defined as a square root of a symmetric non-negative 
$nd\times nd$ square matrix $a^{(1)}(y)$, see \eqref{alinfty}--\eqref{nonneg}. An alternative approach, avoiding directly 
introducing $\left(a^{(1)}(y)\right)^{1/2}$ could be treating in \eqref{apest3} $\nabla u^\varepsilon$ with respect to 
a ($\varepsilon$-rescaled) {\it matrix (tensor) measure} $d\mu_{ijpq}(y)=a^{(1)}_{ijpq}(y)dy$ and appropriately modifying 
further the method of two-scale convergence with respect to measures, cf. \cite{Zhikov2000}. 
\end{remark} 
\begin{proof}
In a standard way, 
setting in (\ref{weak})  $\phi= u^\varepsilon$ results in
\begin{equation}
\int_\Omega \biggl[ a^{(1)}\left(\frac{x}{\varepsilon}\right) {\nabla} u^\varepsilon\cdot \nabla u^\varepsilon(x)+
\varepsilon^2 a^{(0)}\left(\frac{x}{\varepsilon}\right) {\nabla} u^\varepsilon\cdot \nabla u^\varepsilon(x)+ 
\lambda \rho\left(\frac{x}{\varepsilon}\right) u^\varepsilon\cdot u^\varepsilon  \biggr]dx= 
\int_\Omega \rho\left(\frac{x}{\ep}\right)f^\varepsilon(x)\cdot u^\varepsilon(x)\,dx.
\label{apri1}
\end{equation}
The integrals in \eqref{apri1} can be viewed as over the whole of ${\mathbb R}^d$ by extending $u^\ep$ outside  
$\Omega$ by zero. Then one observes that 
the sum of the first two terms and the third term on 
the left hand side 
are non-negative by \eqref{nonneg}--\eqref{coerc1} and \eqref{rhobd},  respectively.  
For the right hand side, 
\[
\int_\Omega f^\varepsilon(x)\cdot u^\varepsilon(x)\,dx\,\leq\,\frac{1}{2}\lambda\nu\left\|u^\varepsilon\right\|_2^2\,
+\,\frac{1}{2}(\lambda\nu)^{-1}\left\|\rho\left(\frac{x}{\ep}\right)f^\varepsilon\right\|_2^2,
\]
which recalling again (\ref{rhobd}) yields (\ref{apest1}). Further, (\ref{coerc1}) and \eqref{nonneg}  immediately imply (\ref{apest2}).
Finally, for the first term on the left hand side of (\ref{apri1}):
\[
\int_\Omega a^{(1)}\left(\frac{x}{\varepsilon}\right) {\nabla} u^\varepsilon\cdot \nabla u^\varepsilon(x)\,dx\,=\,
\left\|\bigl( a^{(1)}(x/\varepsilon)\bigr)^{1/2} \nabla u^\varepsilon \right\|_2^2
\]
which yields (\ref{apest3}).
\end{proof}

\subsection{Functional spaces for two-scale limits}

We next introduce for the periodicity torus $Q$ the following
key linear subspace $V$ of $\left(H^1_{\#}(Q)\right)^n$ of $Q$-periodic (vector-)functions in ${\mathbb R}^d$ which are locally in $H^1$:  
\begin{equation}
V\,:=\, \left\{\biggl. v\in \left(H^1_{\#}(Q)\right)^{n}
\,\biggr\vert \ a^{(1)}(y)\nabla_y v\,=\,0\,\right\}.
\label{herset}
\end{equation}
$V$ can be interpreted as describing the domain of possible
microscopic variations of a (two-scale) limit of the solution $u^\varepsilon$. 

We also introduce  the following `dual' space $W$ of admissible `microscopic fluxes', of tensor fields
on  $Q$:
\begin{equation}
W\,:=\,\left\{\, \psi\,\in\,\left(L^2_\#(Q)\right)^{n\times
d}\,\left\vert \, \mbox{
div}_y\left(\,\left(a^{(1)}(y)\right)^{1/2}\psi(y)\,\right)\,=\,0
\ 
\right.\right\},
\label{wspace}
\end{equation}
where $L^2_\#(Q)$ denotes $Q$-periodic functions from $L^2_{loc}\left({\mathbb R}^d\right)$. 
In \eqref{wspace} the divergence is understood in the sense of distributions on the periodic torus $Q$, i.e., equivalently, 
\begin{equation}
W\,:=\,\left\{\, \psi\,\in\,\left(L^2_\#(Q)\right)^{n\times
d}\,\left\vert \, 
\int_Q  \left(a^{(1)}(y)\right)^{1/2}\psi(y)\cdot \nabla\phi(y)dy\,=\,0, \ \ \forall \phi\in \left(H^1_{\#}(Q)\right)^{n} 
\right.\right\}.
\label{wspace2}
\end{equation}
It immediately follows from the definitions \eqref{herset} and \eqref{wspace2} that $V$ and $W$ are closed linear subspaces of Hilbert spaces $\left(H^1_{\#}(Q)\right)^{n}$ and $\left(L^2_\#(Q)\right)^{n\times
d}$ respectively, and hence can themselves be regarded as Hilbert spaces with respective inherited   
$H_\#^1$ and $L^2$ inner products. 


We will additionally introduce, in a standard way, Hilbert spaces $L^2\left(\Omega; \left(H^1_{\#}(Q)\right)^{n}\right)$, $L^2(\Omega; V)$ and $L^2(\Omega; W)$ of 
functions of two independent variables $x\in\Omega$ and $y\in Q$, which can thereby be regarded as functions of $x$ with values in the appropriate (Hilbert) space. 


The a priori estimates (\ref{apest1})--(\ref{apest3}), via adapting accordingly the properties
of the two-scale convergence, 
imply the following
\begin{lemma}\label{lem22}
Let $\|f^\varepsilon\|_2$ be uniformly bounded. 
Then there exist $u_0(x,y)\in L^2\left(\Omega;\, V\right)$ and
$\xi_0(x,y)\in L^2\left(\Omega; \, W\, \right)$ such that, up
to extracting a subsequence in $\varepsilon$ which we do not
relabel,
\begin{eqnarray}
u^\varepsilon&\stackrel{2}\rightharpoonup& u_0(x,y)         \label{2sc1} \\
\varepsilon \nabla u^\varepsilon  &\stackrel{2}\rightharpoonup& \nabla_y u_0(x,y)           \label{2sc2} \\
\bigl( a^{(1)}(x/\varepsilon)\bigr)^{1/2} \nabla u^\varepsilon\,   &\stackrel{2}\rightharpoonup& \xi_0(x,y).
\label{2sc3}
\end{eqnarray}
\end{lemma}
\begin{proof}
1. According to the theorem on 
(weak) two-scale
compactness of a bounded sequence in $L^2(\Omega)$, 
the \'a priori estimate 
\eqref{apest3} implies, up to extracting a subsequence in
$\varepsilon$ (not relabelled), the existence of a weak 
two-scale limit $\xi_0\in \left(L^2\left(\Omega \times
Q\right)\right)^{n\times d}=L^2\left(\Omega;
\,\left(L^2_{\#}(Q) \right)^{n\times d}\right)$, which yields
\eqref{2sc3}.

We show that in fact $\xi_0(x,y)\in L^2\left(\Omega; \,
W\,\right)$. Take in \eqref{weak}
$\phi(x)=\phi^\varepsilon(x)=\,
\varepsilon\,\varphi(x)b\left(\frac{x}{\varepsilon}\right)$ for
any $\varphi\in C_0^\infty(\Omega)$ and $b\in \left(C_\#^\infty(Q)\right)^{n}$. 
Passing then to the limit
in \eqref{weak} we notice, via \eqref{apest1} and
\eqref{apest2}, that the limit of each term but the first one
on the left hand-side of \eqref{weak} is zero, and therefore
\begin{eqnarray}
\lim_{\varepsilon \to 0} \int_\Omega
a^{(1)}\left(\frac{x}{\varepsilon}\right) \nabla
u^\varepsilon(x)\,\cdot\,\varepsilon\nabla\left( \varphi(x)b\left(
\frac{x}{\varepsilon}\right)\right) \,dx\,=\, \ \ \ \ \ \ 
\nonumber\\ 
\ \ \ \ \ \ \ \ 
\int_\Omega\varphi(x)\int_Q 
\left(a^{(1)}(y)\right)^{1/2}
\xi_0(x,y)\,\cdot\,\nabla_y b(y)\,dy\,dx\,=\,0, 
\label{pfl221o}
\end{eqnarray}
where we have used the assumption \eqref{alinfty} of boundedness of $a^{(1)}$. 
The density of $\varphi(x)$ in $L^2(\Omega)$ implies that for all $b\in \left(C^\infty_\#(Q)\right)^n$ the  inner integral is zero for a.e. $x\in\Omega$. Since $b(y)$ are in turn dense in 
$ \left(H^1_{\#}(Q)\right)^{n}$, this implies that, for a.e. $x$, $\xi_0(x,\cdot)$ obeys \eqref{wspace2} 
and hence $\xi_0(x,\cdot)\in W$ implying $\xi_0\in L^2\left(\Omega;\,W\right)$.

2. Further, according to e.g. \cite[Prop. 1.14 (ii)]{allaire}, $ \ $ \eqref{apest1} together with \eqref{apest2}
imply \eqref{2sc1}--\eqref{2sc2}
for some $u_0(x,y)\in L^2\left(\Omega;\, \left(H^1_\#(Q)\right)^n\right)$. 

Show finally that in fact $u_0(x,y)\in L^2\left(\Omega;\,
V\right)$. For any $\psi(x,y)=\varphi(x)b\left(\frac{x}{\varepsilon}\right)$ with 
$\varphi\in C_0^\infty(\Omega)$ and $b\in \left(C_\#^\infty(Q)\right)^{n\times d}$, 
\begin{equation}
\lim_{\varepsilon \to 0} \int_\Omega \left(a^{(1)}\left(\frac{x}{\varepsilon}\right)\right)^{1/2}
\varepsilon\nabla u^\varepsilon(x)\cdot
\psi\left(x, \frac{x}{\varepsilon}\right) \,dx\,=\,
\int_\Omega\int_Q \left(a^{(1)}(y)\right)^{1/2}
\nabla_y u_0(x,y)\cdot\psi(x,y)\,dx\,dy,
\label{pfl221}
\end{equation}
where we have used \eqref{2sc2}. 

On the other hand, \eqref{apest3} ensures that
\[
\left\|\left(a^{(1)}\left(\frac{x}{\varepsilon}\right)\right)^{1/2}\varepsilon\nabla u^\varepsilon(x)\right\|_2\,\to\,0,
\]
and hence the limit in \eqref{pfl221} is zero. This implies for the right hand side of \eqref{pfl221},
\[
\int_\Omega\varphi(x)\int_Q \left(a^{(1)}(y)\right)^{1/2}\nabla_y u_0(x,y)\cdot b(y)\,dy\,dx\,=\,0,
\ \ \ \forall \varphi \in C_0^\infty(\Omega), \ b\in \left(C_\#^\infty(Q)\right)^{n\times d}.
\]
By density of $\varphi$ and $b$, this gives
\begin{equation}
\left(a^{(1)}(y)\right)^{1/2}\nabla_y u_0(x,y)\,=\,0 \ \ \mbox{for a.e. } \, x,
\label{pfl222}
\end{equation}
and therefore, pre-multiplying \eqref{pfl222} by $\left(a^{(1)}(y)\right)^{1/2}$,
yields $u_0(x,y)\in L^2\left(\Omega;\, V\right)$, cf \eqref{herset}.
\end{proof}

\section{A generic class of degeneracies and related properties.}

A key problem in the homogenization theory is to relate the limit ``fluxes'' (in the
present case the ``modified'' fluxes $\xi_0(x,y)$) to the limit ``fields''
$u_0(x,y)$. This can be achieved only if  the general degeneracy described by $a^{(1)}(y)$
satisfies some additional restrictions. We impose below a key generic
technical assumption on $a^{(1)}(y)$ which is sufficient for this purpose.
We will 
see that this
assumption {\it is} satisfied for most of previously studying models in both
classical and non-classical homogenization, as well as will refer to some more recent examples.

Let $\left( \,\cdot\,,\, \cdot \,\right)_{H^1}$ be an inner
product in $\left(H_\#^1(Q)\right)^n$. Denote $V^\bot$ the
orthogonal complement  to $V$ defined by \eqref{herset}, i.e.
\[
 V^\bot\,:=\,\left\{\biggl. w\in \left(H_\#^1(Q)\right)^n\, \biggr\vert\,
 \left(\, w, \, v\,\right)_{H^1}\,=\,0, \ \ \forall
v\in V  \right\}.
\]
Then $\left(H_\#^1(Q)\right)^n$ is  a direct orthogonal sum
of (closed)  $V$ and $V^\bot$,
\begin{equation}
\left(H_\#^1(Q)\right)^n\,=\,V\,\oplus\,V^\bot,
\label{vort}
\end{equation}
 i.e. any $v$ in $\left(H_\#^1(Q)\right)^n$ is
uniquely decomposed into the sum $v=v_1+v_2$, where $v_1\in V$ and $v_2\in V^\bot$.
The key assumption is the following:

\begin{assumption}[Key assumption on the degeneracy]:   
{\it There exists a constant $C>0$ such that for all $v\in \left(H_\#^1(Q)\right)^n$
there exists $v_1\in V$ with
\begin{equation}
\left\|v\,-v_1\right\|_{\left(H^1_\#(Q)\right)^n}\,\leq\, C\,\left\|a^{(1)}(y)\nabla_yv\right\|_2.
\label{keyass}
\end{equation} }
\vspace{.1in}
\label{keyassump}
\end{assumption} 

The condition \eqref{keyass} can 
be equivalently re-written as
\begin{equation}
\left\|P_{V^\bot} v\right\|_{\left(H^1_\#(Q)\right)^n}\,\leq\, C\,\left\|a^{(1)}(y)\nabla_yv\right\|_2,
\label{keyass2}
\end{equation}
where $P_{V^\bot}$ is the orthogonal projector on $V^\bot$. The equivalence of \eqref{keyass}
and \eqref{keyass2} immediately follows by noticing that $v_1=P_Vv$, where $P_V$ is the orthogonal
projector on $V$, is the best (i.e. minimizing the left hand side of \eqref{keyass}) choice of $v_1$ for \eqref{keyass}. 
The property \eqref{keyass}, and hence 
equivalently  \eqref{keyass2}, obviously does not depend on the choice of the (equivalent)
inner product and hence the norm in $\left(H_\#^1(Q)\right)^n$.

\subsection{Examples of the key assumption \eqref{keyass}}\label{keyassexamples}
The key assumption \eqref{keyass}, equivalently \eqref{keyass2}, as well as indeed the initial assumptions 
\eqref{rhobd}--\eqref{coerc1}, 
appears to hold for most 
of the previously considered cases which may involve an $\ep^2$-contrast of a general form \eqref{aeps}. In each particular case, the validity or otherwise of \eqref{keyass} has to be established by separate means, and we briefly discuss below some of those cases. 

1. {\it Classical scalar homogenization}. In this simplest case, $n=1$, 
$\rho_{11}(y)\equiv 1$, and $a^{(1)}_{1j1q}(y)$ is a uniformly bounded symmetric 
positive definite matrix i.e. $a^{(1)}_{1j1q}(y)\zeta_j\zeta_q\geq \nu|\zeta|^2$ for all 
$\zeta\in {\mathbb R}^d$ and for a.e. $y\in Q$; $a^{(0)}_{1j1q}(y)\equiv 0$. 
Then, from \eqref{herset}, $V$ is the 
one-dimensional space of constant 
functions on $Q$ and \eqref{keyass} follows from the Poincar\'{e} inequality 
in $H^1_\#(Q)$: 
$\left\|v\,-\langle v\rangle\right\|_{L^2_\#(Q)}\,\leq\, C\,\left\|\nabla_yv\right\|_{L^2_\#(Q)}$ for all 
$v\in H^1_\#(Q)$, where $\langle v\rangle:=\int_Q v$ is the mean of $v$. Hence \eqref{keyass} 
immediately follows by taking $v_1=\langle v\rangle\in V$. 
Notice that the nonegativity and uniform strong ellipticity conditions \eqref{nonneg} and 
\eqref{coerc1} are also trivially held. 

2. {\it Double porosity-type models.} This corresponds, in the simplest case (see e.g. \cite{Zhikov2000}), 
to $n=1$ and $a^{(1)}_{1j1q}(y)=\delta_{jq}\chi_1(y)$, $a^{(0)}_{1j1q}(y)=\delta_{jq}\chi_0(y)$, 
where $\delta_{ij}$ is Kroneker symbol, 
$\chi_1(y)=1-\chi_0(y)$, and $\chi_0(y)$ is characteristic function of an 
 open inclusion $Q_0$, $\overline{Q_0}\subset Q$ with regular enough boundary and connected complement 
$Q_1:=Q\backslash \overline{Q_0}$. (Hence $\chi_1$ is 
characteristic function of a periodically connected matrix $Q_1$.) 
According to \eqref{herset}, $v\in V\subset H^1_\#(Q)$ must be constant on 
$Q_1$ and 
arbitrary otherwise, i.e. for any $v\in V$, $v=c+\tilde v$ where $c\in{\mathbb R}^d$ and 
$\tilde v\in H^1_0(Q_0)$ extended to $Q_1$ by zero. 
The key assumption \eqref{keyass} then directly follows from an extension lemma (see e.g. \cite{JKO}, Lemma 3.2) 
implying (in particular) that given $v\in H_\#^{1}(Q)$ there exists $ v_2\in H_0^1(Q_0)$ 
such that 
$\left\|\nabla(v\,-v_2)\right\|_{L^2_\#(Q)}\,\leq\, C\,\left\|\nabla v\right\|_{L^2_\#(Q_1)}$. 
Then, by choosing $v_1=v_2+\langle v-v_2\rangle\in V$, \eqref{keyass} follows from the  Poincar\'{e} inequality 
and the above extension result. Note that the conditions \eqref{nonneg} and 
\eqref{coerc1} are again trivially held. 

One can see that the assumption \eqref{keyass} is in fact satisfied for rather general ``multi-component'' 
high-contrast configurations, cf. e.g. \cite{panasenko, Zhikov2000}, of the double porosity type. For example, for $d\geq 3$, let $a^{(1)}_{1j1q}(y)=\delta_{jq}\sum_{m=1}^M\chi_m(y)$, where $\chi_m$, $m=1,2,...,M$, are characteristic 
functions of disjoint ``stiff'' phases $Q_m$ each of which is periodically connected and has Lipschitz boundary. In the remaining ``soft'' 
phase $Q_0=Q\backslash \cup_{m=1}^M \overline{Q_m}$, let $a^{(0)}_{1j1q}(y)=\delta_{jq}\chi_0(y)$. Then 
$V$ consists of all $v\in H^1_\#(Q)$ whose values on $Q_m$ are some constants $c_m\in {\mathbb R}$. Given 
$v\in H^1_\#(Q)$, a function $v_1$ satisfying \eqref{keyass} can be constructed as follows. Set 
$c_m=\langle v\rangle_m:=|Q_m|^{-1}\int_{Q_m} v(y) dy$ (i.e. $c_m$ is the mean of $v$ over $Q_m$), and let 
$\hat v(y)=v(y)-\sum_{m=1}^M c_m\chi_m(y)$. Let 
$\tilde v=S \hat v$ where $S$ is an $H^1$-extension from the `combined' stiff phase $Q_s:=\cup_{m=1}^M Q_m$ to 
$H^1_\#(Q)$ i.e. 
$S w(y)=w(y)$ for $y\in Q_s$ and 
$\|S w\|_{H^1_\#(Q)}\leq C\|w\|_{H^1_\#(Q_s)}$\footnote{
The extension theorems, see e.g. \cite{evans,stein}, are normally formulated for Euclidean domains rather than 
for a periodic torus as needed here. However the result of e.g. Theorem 5 of \S VI.3 of Stein \cite {stein} can 
be used to deduce the desired statement. For example, 
consider an extension from 
$Q_s\subset {\mathbb R}^d$ which is regarded as an infinite (periodic) set. Then it satisfies all the conditions 
from the above theorem of Stein. Take an infinitely periodic 
$w\in H^1_\#(Q_s)\subset H^1_{loc}({\mathbb R}^d)$ and multiply it by a smooth cut-off function 
$\chi_R(y)$ such that $\chi_R=1$ for $|y|\leq R$, $\chi_R=0$ for $|y|\geq R+1$ and $|\nabla\chi_R(y)|\leq C$. 
Apply the Stein's theorem to $\chi_R w$, denote the relevant extension by 
$\tilde w_R(y)
$ and consider its (normalized) `periodization' 
$w_R(y):=|B_R|^{-1}\sum_{k\in{\mathbb Z}^d}\tilde w_R(y+k)$. One 
 readily checks that $\|w_R-w\|_{H^1(Q_s)}\to 0$ as $R\to\infty$, and then by continuity of the Stein's extension that $w_R$ has a 
limit ${\cal S} w$ in $H^1_\#(Q)$ with $\|{\cal S}w\|_{H^1_\#(Q)}\leq C\| w\|_{H^1_\#(Q_s)}$ and so can be 
taken as the desired extension.
}. 
We then set 
$v_1(y)=v(y)-\tilde v$. It is readily checked that $v_1(y)=c_m$ on $Q_m$, $1\leq m\leq M$, and so 
$v_1\in V$. Further 
\[
\|v-v_1\|_{H^1_\#(Q)}=\left\|
S\hat v\right\|_{H^1_\#(Q)}\leq 
C
\left\|\hat v\right\|_{H^1_\#(Q_s)}\leq \ \ \ \ 
\]
\[
\ \ \ \ \ \ \ \ \ \ \ \ 
C\sum_{m=1}^M\left\| v(y)-c_m\right\|_{H^1_\#(Q_m)}\leq C \|a^{(1)}(y)\nabla_y v\|_2, 
\]
which gives \eqref{keyass}. (In the last step we have applied the Poincar\'{e} inequality for each $Q_m$, 
noticing that $\langle v-c_m\rangle_m=0$.) 

Similar extension arguments apply to the cases of `isolated' stiff components, e.g. when $Q_1$ is an inclusion, 
$\overline{Q_1}\subset Q$, $Q_0=Q\backslash \overline{Q_1}$.


3. {\it Classical homogenization for linear elasticity.} Let $n=d=3$, $a^{(0)}_{ijpq}(y)\equiv 0$, and 
\beq
a^{(1)}_{ijpq}(y)=\lambda(y)\delta_{ij}\delta_{pq}+
\mu(y)\left(\delta_{ip}\delta_{jq}+\delta_{iq}\delta_{jp}\right), 
\label{lame}
\eeq 
with Lam\'{e} coefficients $\lambda,\mu\in L^\infty_\#(Q)$ such that $\mu(y)\geq\mu_0>0$ and 
$\lambda(y)+2\mu(y)/3\geq \kappa_0>0$. 

Then $a^{(1)}(y)\nabla_yv\equiv 0$ implies that $v$ is a rigid body displacement 
(translation and/ or rotation),  and since the periodicity condition excludes rotations $V$ as defined by \eqref{herset} can only contain 
translations i.e. constant vector functions. Then one can see that \eqref{keyass} holds with 
$v_1=\langle v\rangle$ due to the periodic Korn inequality. 
Condition \eqref{nonneg} is 
known to hold and is equivalent to non-negativity of elastic energy density. Finally \eqref{coerc1} follows by e.g. 
bounding the integrand on its left hand side from below by replacing $a^{(1)}(y)$ with its `homogeneous' analog where 
$\lambda(y)$ and $\mu(y)$ in \eqref{lame} are replaced by, respectively, 
$\lambda_0=\kappa_0-2\mu_0/3$ and $\mu_0$. The resulting integral still satisfies \eqref{coerc1}, which can be 
shown e.g. via applying Fourier transform and Plancherel's theorem. 

We emphasize that the present approach does not cover all the cases of strong ellipticity for linear elasticity. 
For example, for constant $\lambda$ and $\mu$, $\lambda(y)=\lambda_0$ and $\mu(y)=\mu_0$, 
the condition 
ensuring (in the absence of $a^{(0)}$) \eqref{coerc1} is known to be $\mu_0\geq\nu$ and $\lambda_0+2\mu_0\geq \nu$, $\nu>0$. So for $\mu_0>0$ and $-2\mu_0<\lambda_0<-2\mu_0/3$, the condition \eqref{nonneg} would not hold. 
Remark that under certain scenarios the `strict strong ellipticity' in linear elasticity can be lost 
through 
homogenization, cf. e.g. \cite{briane}.

4. {\it Elasticity, soft inclusions, cf \cite{ZhP13}.} Let $n=d=3$, and given an inclusion $Q_0$ as in Example 2 above, 
$a^{(1)}(y)$ be as in \eqref{lame} but additionally multiplied by $\chi_1(y)$, the characteristic function of 
connected matrix $Q_1=Q\backslash\overline{Q_0}$. Let $a^{(0)}(y)$ be also as in \eqref{lame} multiplied in 
turn by $\chi_0(y)=1-\chi_1(y)$. (So the model is a linear elastic version of the above double porosity one.) 
Then $V=\left\{v\in \left(H^1_\#(Q)\right)^3\,:\, v=c+\tilde v, c\in {\mathbb R}^3, 
\tilde v\in \left(H^1_0(Q_0)\right)^3\right\}$, and \eqref{keyass} can be achieved e.g. by combining the above periodic 
Korn inequality with an 
extension lemma. (One way for achieving such an extension is essentially as in the second half of 
Example 2, i.e. by setting $v_1=v-S(v-\langle v\rangle_1)$ where $\langle v\rangle_1$ is the mean of $v$ over 
the matrix $Q_1$ and $S$ is an $H^1_\#$-bounded extension from $Q_1$ to $\left(H^1_\#(Q)\right)^3$.) 
Similarly to the previous example, conditions \eqref{nonneg} 
and \eqref{coerc1} are checked to be readily satisfied. 

Similarly to Example 2, one can show that \eqref{keyass} holds also for multi-component elastic 
matrices with connected stiff components.

5. {\it Elasticity with $O(\ep^2)$ shear modulus in inclusions \cite{Cooper14}}. 
In this case $a^{(1)}(y)$ is as in \eqref{lame} except $\mu(y)$ is additionally multiplied by $\chi_1(y)$, and 
$a^{(0)}(y)$ is in turn as in \eqref{lame} multiplied by $\chi_0(y)$. (So the inclusions is stiff in compression but soft in shear.) Then, assuming $\overline{Q_0}\in Q$ and $\partial Q_0$ regular enough, 
\[
V=\left\{v\in \left(H^1_\#(Q)\right)^3\,:\, v=c+\tilde v, c\in {\mathbb R}^3, 
\tilde v\in \left(H^1_0(Q_0)\right)^3; \, \, \mbox{div\,} v=0 \,\mbox{ in }\, Q\right\}. 
\]
Then, as shown in \cite{Cooper14}, the key assumption \eqref{keyass} follows from a `modification lemma' 
(a version of a lemma on existence of vector fields with prescribed divergence, see e.g. \cite{piatn}): 
given a vector filed in $H^1_\#(Q)$, there exists $v_2\in \left(H^1_0(Q_0)\right)^3$ such that $\mbox{div}\,v_2=0$ in $Q$, 
and 
\[
\left\|\nabla(v\,-v_2)\right\|_{\left(L^2_\#(Q)\right)^3}\,\leq\, C\,
\left(\left\|\nabla v\right\|_{\left(L^2(Q_1)\right)^{3\times 3}}+ \left\|\mbox{div\,} v\right\|_{L^2(Q_0)}\right). 
\]

6. {\it Elasticity with stiff fibers.} In this case some stiff components can allow certain periodic rotations, cf 
\cite{bellieud}. 
In the simplest case of a single stiff cylindrical fiber, the equations have the same form as in Example 4, but $Q_1$ is a cylinder, i.e. $Q_1= \hat{Q}_1\times (0,1)$, where the two-dimensional connected cross-section $\hat{Q}_1$ with smooth boundary is such that $\overline{\hat{Q}_1}\subset (0,1)^2$. Then $V=\left\{v\in \left(H^1_\#(Q)\right)^3\,:\, v=c+\alpha\,y\times e_3 \, \text{in}\, Q_1 \right\}$, where $ c\in {\mathbb R}^3$, $\alpha\in {\mathbb R}$, and $\times$ denotes the standard vector cross-product. Here $c+\alpha\, y\times e_3$ represent  admissible (i.e. consistent with the $Q$-periodicity condition) rigid body motions of $Q_1$, i.e arbitrary translations and a rotation about the cylinder's axis parallel to the unit vector $e_3$ in the 
$y_3$-direction. In order to verify \eqref{keyass}, for a given $v\in \left(H^1_\#(Q)\right)^3$ define ${\tilde v}\in \left(H^1_\#(Q_1)\right)^3$ by ${\tilde v}=v-\tilde c-\tilde\alpha \,y\times e_3$, where 
$ \tilde c\in {\mathbb R}^3$  and $\tilde\alpha\in {\mathbb R}$  are such that 
$$ \int_{Q_1}{\tilde v}\,dy=0\ \text{and} \int_{Q_1}{\tilde v}\cdot\left( y\times e_3\right)dy=0\,,
$$
i.e. $\tilde v$ has zero average translations and rotations. (It is straightforward to see that such, unique, 
$\tilde c$ and $\tilde\alpha$ do exist.) 
Then one can choose $v_1$ in  \eqref{keyass} as follows: 
$v_1(y)=v(y)-
(S{\tilde v} )(y)$ where 
$S:\left(H^1_\#(Q_1)\right)^3\rightarrow \left(H^1_\#(Q)\right)^3$ is any 
$H^1$-bounded extension. Indeed $v_1\in V$, and 
$$\|v-v_1\|_{\left(H^1_\#(Q)\right)^3}= \|
S{\tilde v}\|_{\left(H^1_\#(Q)\right)^3}\leq C 
\|{\tilde v}\|_{\left(H^1_\#(Q_1)\right)^3}.
$$
It remains to employ the following version of Korn's inequality
$$C\|{ w}\|_{\left(H^1_\#(Q_1)\right)^3}^2\leq  \int_{Q_1 }a^{(1)}\nabla w \cdot \nabla w \,dx \,+\,
\left|\int_{Q_1}{\tilde w}\,dy\right|^2\,+\,\left|\int_{Q_1}{\tilde w}\cdot\left( y\wedge e_3\right)dy\right|^2,\, \forall w \in \left(H^1_\#(Q_1)\right)^3.
$$
The latter 
in turn follows from the standard Korn's inequality in $H^1(Q_1)$ and usual arguments about equivalent norms in Banach spaces, see e.g. {\it equivalence lemma} in  \cite{tartar}. 

Similar arguments apply to the cases of presence of several stiff fibers parallel to different axes and/ or 
of isolated stiff `grains' (with unconstrained rotations for the latter), cf. \cite{bellieud}.

7. {\it Photonic crystal fibers with a `near-critical' propagation}. 
As shown in \cite{CKS14}, for a photonic crystal fiber type waveguide structure and the wave propagation with an $e^{i\beta x_3}$-dependence in the Maxwell's equations along the fibers, for 
an `almost critical' propagation constants $\beta$ 
the problem 
can be reduced to that of the form 
\eqref{rhobd}--\eqref{aeps} with $n=d=2$, 
$\rho(y)=\chi_1(y)\rho_1+\chi_0(y)\rho_0$ with certain constant positive diagonal 2$\times$2 matrices 
$\rho_0$ and $\rho_1$, 
 and with a degenerate 
quadratic form due to $a^{(1)}(y)$ as follows: 
\[
a^{(1)}(y)\nabla v\cdot\nabla v\,=\, 
\chi_1(y)\left(
|v_{1,1}+v_{2,2}|^2+\,|v_{1,2}-v_{2,1}|^2\right), 
\]
with $\chi_1(y)$ being characteristic function of (two-dimensional) connected matrix $Q_1$, 
$\chi_0(y)=1-\chi_1(y)$,  similarly to the previous examples. 
Then $V=\left\{ v\in \left(H^1_\#(Q)\right)^2\,:\, v_{1,1}+v_{2,2}=v_{1,2}-v_{2,1}=0 \ \mbox{ in } Q_1\right\}$
i.e. $v$ is required to satisfy   Cauchy-Riemann type conditions in $Q_1$. 

The key assumption \eqref{keyass} then states that there exists $v_1\in V$ such that 
\[
\left\|v\,-v_1\right\|_{\left(H^1_\#(Q)\right)^2}\,\leq\, C\,\left(\left\| v_{1,1}+v_{2,2}\right\|_{L^2(Q_1)}+ \left\|v_{1,2}-v_{2,1}\right\|_{L^2(Q_1)}\right). 
\]
The latter inequality is proved in \cite{CKS14}. 

8. {\it Three-dimensional Maxwell equations with high contrast in electric permittivity (
cf. \cite{CC15} and \cite{BBF17}).}
When the electric permittivity is of order $\ep^{-2}$ in an inclusion and of order one in a (simply connected) 
matrix $Q_1$, the problem can be reduced to the following case:  
\[
V=\left\{v\in\left(H^1_\#(Q)\right)^3: \, \mbox{div\,} v=0 \mbox{  in } Q; \,\mbox{curl\,} v=0 \mbox{  in } Q_1 \right\}.
\] 
Then the key assumption 
\eqref{keyass} can be reduced to the following: given $v\in\left(H^1_\#(Q)\right)^3$ there exists 
$v_1\in V$ 
 such that 
\beq
\left\|\nabla(v\,-v_1)\right\|_{\left(L^2_\#(Q)\right)^3}\,\leq\, C\,
\left( \left\| \mbox{curl\,} v\right\|_{\left(L^2(Q_1)\right)^3}+\left\| \mbox{div\,} v\right\|_{L^2(Q)}
\right).
\label{kassmaxw}
\eeq
See \cite{CC15} where \eqref{kassmaxw} was essentially proved for some related details, as well as recent work \cite{BBF17} on a related topic. 

9. {\it ${\cal A}$-quasiconvexity constant rank assumption}, cf. \cite{FM}.  If 
$a^{(1)}(y) \equiv a^{(1)}$ is a constant tensor, i.e. it does not depend on $y$ at all (which formally still keeps it 
in the general class of periodic functions), then for $v\in \left(H^1_\#(Q)\right)^n$ the first order 
linear differential operator with constant coefficients ${\cal A}v:= a^{(1)}\nabla v$ may be viewed as a 
`differential constraint' and its null space $V:=\left\{v\in \left(H^1_\#(Q)\right)^n: {\cal A}v=0 \right\}$ determines the set of oscillating (periodic) vector fields subject to this differential constraint. 
Then, applying Fourier transform in $Q$-periodic $y$, it is not hard to see that 
\eqref{keyass} is valid if and only if the $n\times n$ matrix $\tilde A(\xi)$, 
$\tilde A_{ip}(\xi):=a^{(1)}_{ijpq}\xi_j\xi_q$ has a constant rank for all $\xi\in S^{d-1}$ where $S^{d-1}$ is 
the unit sphere in ${\mathbb R}^d$. 
This is  similar to constant rank condition in \cite{FM}, which in turn implies a similar key  
decomposition property in the 
$ {\cal A}$-quasiconvexity theory ensuring lower semi-continuity for appropriate variational functionals 
subject to differential constraint $ {\cal A}v:=a^{(1)}\nabla v=0$.  
\vspace{.1in} 

The above list of examples could be continued. 
As a trivial example, it includes the case of 
$a^{(1)}(y)\equiv 0$ (with $a^{(0)}(y)$ hence uniformly strongly elliptic in the sense of \eqref{coerc1}). 
Then obviously $V=\left(H^1_\#(Q)\right)^n$, with the key assumption \eqref{keyass} trivially held 
with $v_1=v$. 

As an example when \eqref{keyass} is {\it not} satisfied, we mention the case of 
highly anisotropic fibers studied in \cite{CSZ06}: 
let $n=1$, $d=3$, and $a^{(1)}_{1j1q}(y)=\delta_{jq}\chi_1(y)+\delta_{j3}\delta_{q3}\chi_0(y)$, 
where $\chi_0$ is the characteristic function of a cylinder $Q_0=\hat Q_0\times [0,1)$, 
$\overline{\hat Q_0}\subset (0,1)^2$. Then for $v\in V$, $v(y)=c+\tilde v(\tilde y)$, where 
$c\in {\mathbb R}$, $\tilde v\in H_0^1(\hat Q_0)$, $\tilde y:=(y_1,y_2)$. One can then see by, 
for example, fixing $v_0(\tilde y)\in H_0^1(\hat Q_0)$, setting 
$v_n(y):= v_0(\tilde y)\sin(n y_1)\cos(2\pi y_3)$ and then increasing $n$ that \eqref{keyass} could not 
possibly be satisfied. For similar reasons, \eqref{keyass} does not appear to be satisfied in the 
two examples with elastic high anisotropy in \S 5 of \cite{VPS09}. 
Notice that \cite{CSZ06} nevertheless establishes a version of the two-scale 
resolvent convergence by employing additional ideas due to two-scale convergence with respect to 
measures, cf. \cite{Zhikov2000} which is likely to be applicable also to the examples in \cite{VPS09}, as well as that in other examples involving ``partial degeneracy'' 
the key assumption does hold, e.g. \cite{Cooper14,CKS14}.


\subsection{Properties under the key assumption \eqref{keyass}} 

The condition \eqref{keyass} implies a number of important properties as we demonstrate below.
First it allows to formulate an appropriate well-posed version of the unit cell corrector problem.
We state a related fact in some generality as follows.

Consider the following degenerate boundary value problem on the periodicity cell $Q$:
\begin{equation}
-\,\mbox{div}_y\left(a^{(1)}(y)\nabla_y v\right)\,=\,F, \ \ \ v\in  \left(H_\#^1(Q)\right)^n,
\label{corrgen}
\end{equation}
where $F\in \left(H_\#^{-1}(Q)\right)^n$ is given (i.e. $F$ by definition is a linear continuous functional
on $\left(H_\#^{1}(Q)\right)^n$). For arbitrary $G\in \left(H_\#^{-1}(Q)\right)^n$ and
$w\in \left(H_\#^1(Q)\right)^n$ we denote by $\left\langle G, w\right\rangle$
the duality action of $G$ on $w$. The problem \eqref{corrgen} is then equivalently
formulated in a weak form as follows: find $v\in \left(H_\#^1(Q)\right)^n$ such that
\begin{equation}
\int_Q\, a^{(1)}(y)\nabla_y v(y)\cdot \nabla_y w\,dy\,=\,\langle\, F\,,\, w\,\rangle,
\ \ \ \forall \ w\in \left(H_\#^1(Q)\right)^n.
\label{corrgenvar}
\end{equation}

\begin{theorem}\label{thmsolv}
Under the assumption \eqref{keyass}, 

(i) The problem \eqref{corrgen}, equivalently \eqref{corrgenvar}, is solvable in
$\left(H_\#^1(Q)\right)^n$ if and only if
\begin{equation}
\left\langle\, F\,,\, w\,\right\rangle\,=\,0, \ \ \ \forall\, w\in V.
\label{solvmain}
\end{equation}
When \eqref{solvmain} does hold, the problem \eqref{corrgen} or \eqref{corrgenvar}
is uniquely solvable in $V^\bot$. \newline
(ii) For any solution $v$ and any $v_1\in V$, $v+v_1$ is also a solution. Conversely,
any two solutions can only differ for $v_1\in V$.
\end{theorem}
\begin{proof}
(i)
Let $v$ be a solution of \eqref{corrgenvar} and let $w\in V$. Then,
using the symmetry of $a^{(1)}$ and \eqref{herset},
\begin{equation}
\left\langle F\,,\, w\right\rangle\,=\,
 \int_Q\, a^{(1)}(y)\nabla_y
v(y)\cdot \nabla_y w\,dy\,=\, \int_Q\, \nabla_y v(y)\cdot
a^{(1)}(y)\nabla_y w\,dy\,=\,0 \label{fwsymm}
\end{equation}
 yielding
\eqref{solvmain}. Conversely, let \eqref{solvmain} hold and
seek $v\in \left(H_\#^1(Q)\right)^n$ solving
 \eqref{corrgenvar}.
 By \eqref{fwsymm}, the identity \eqref{corrgenvar} is automatically held for
all $w$ in $V$, therefore it is sufficient to verify it for all $w\in V^\bot$.
Seek $v$ also in $V^\bot$. Show that then, in the Hilbert space $H:=V^\bot$ with
the $\left(H_\#^1(Q)\right)^n$-inherited norm $\|\cdot\|_H$,
the problem   \eqref{corrgenvar} satisfies the conditions of the Lax-Milgram lemma
(see e.g. \cite{evans}). Namely, first the bilinear form
\[
B[v,w]\,:=\,\int_Q\, a^{(1)}(y)\nabla_y v(y)\cdot \nabla_y w\,dy
\]
is immediately shown via \eqref{alinfty} to be bounded in $H$, i.e. with some $C>0$,
\[
\biggl\vert\, B[\,v\,,\,w\,] \,\biggr\vert\,\leq\, C\,
\|v\|_H\,\|w\|_H, \ \ \ \forall v,\, w\,\in\,H.
\]
Show now that the form $B$ is coercive, i.e. for some $\nu>0$,
\[
B[v,v]\,\geq \, \nu\, \|v\|_{H}^2, \ \ \ \forall v \in V^\bot.
\]
We have
\[
B[v,v]:=\int_Q\, a^{(1)}(y)\nabla_y v(y)\cdot \nabla_y v\,dy=
\left\|\left(a^{(1)}(y)\right)^{1/2}\nabla_yv\right\|_2^2\,\geq\,
\]
\[
C\,\left\|\,a^{(1)}(y)\nabla_yv\,\right\|_2^2\,\,\geq\,\,
\nu\, \|v\|_H^2, 
\]
with some $\nu>0$. 
In the last two inequalities we have used, respectively, \eqref{alinfty} and
\eqref{keyass2}.

Therefore, by the Lax-Milgram lemma, there exists a unique solution to the problem
\[
v\in V^\bot: \ \ B[v,w]\,=\,\langle\, F\,,\,w\,\rangle, \ \ \ \forall w\,\in\,V^\bot,
\]
and hence to \eqref{corrgenvar}.

(ii) If $v$ solves \eqref{corrgenvar} and $v_1\in V$ then
$a^{(1)}(y)\nabla_yv_1(y)=0$ and hence $v+v_1$ also solves
\eqref{corrgenvar}.

Assuming further $v^{(1)}$ and $v^{(2)}$ both solve \eqref{corrgenvar}, set
$v=v^{(1)}-v^{(2)}$ solving hence \eqref{corrgenvar} with $F=0$, and then set
$w=v$. As a result,
\[
0\,=\,\int_Q\, a^{(1)}(y)\nabla_y v(y)\cdot \nabla_y v\,dy\,=\,
\left\|\left(a^{(1)}(y)\right)^{1/2}\nabla_yv\right\|_2^2,
\]
implying $\left(a^{(1)}\right)^{1/2}\nabla_yv\,=\,0$ and hence 
$a^{(1)}\nabla_yv\,=\,0$, i.e. $v\in V$.
\end{proof}

Recalling definition \eqref{wspace2} of the dual space $W$, 
the next important property is a generalization of the Weyl's decomposition, 
cf. e.g. \cite{JKO,piatn}, for degenerate $a^{(1)}$
satisfying \eqref{keyass}.

\begin{theorem}\label{weyldec}
Let $a^{(1)}$ satisfy \eqref{keyass}, and let $\eta\in \left(L^2_\#(Q)\right)^{n\times d}$.
Suppose $\eta$ is orthogonal in $\left(L^2_\#(Q)\right)^{n\times d}$ to $W$, i.e.
\begin{equation}
\left(\,\eta\,,\, \psi\,\right)_2\,:=\,\int_Q
\eta_{ij}(y)\psi_{ij}(y)dy\,=\,0, \ \ \ \forall \,
\psi\,\in\,W. \label{ortw}
\end{equation}
Then there exists $u_1\in\,\left(H_\#^{1}(Q)\right)^n$ such that
\begin{equation}
\eta(y)\,=\,\left(a^{(1)}(y)\right)^{1/2}\nabla_yu_1(y).
\label{weylrep}
\end{equation}
Such a $u_1$ is determined  uniquely up to an arbitrary function from $V$, in particular
is unique in $V^\bot$.
\end{theorem}
\begin{proof}
Let $\eta$ satisfying \eqref{ortw} be given, and seek $u_1$ such that \eqref{weylrep} holds. 
For $w\in \left(H^1_\#(Q)\right)^n$, multiply \eqref{weylrep} by $\left(a^{(1)}(y)\right)^{1/2}\nabla w(y)$ 
and integrate over $Q$. 
As a result,
\begin{equation}
\int_Q a^{(1)}(y)\nabla_y u_1\cdot\nabla w \,dy \,=\, 
\int_Q \left(a^{(1)}(y)\right)^{1/2}\eta(y)\cdot \nabla w\,dy\,=:\,\langle F,w\rangle. 
\label{kuka1}
\end{equation}
Check that the above defined $F\in
\left(H_\#^{-1}(Q)\right)^n$ satisfies the condition
\eqref{solvmain}. Indeed, 
\beq
\langle\,F\,,\,w\,\rangle\,=\,
\int_Q \eta(y)\cdot \left(a^{(1)}(y)\right)^{1/2}\nabla w(y)\,dy, 
\label{aksk}
\end{equation}
and so if $w\in V$ it follows that $a^{(1)}(y)\nabla w(y)=0$ for a.e. $y$,
and hence  (for a.e. $y$), $\left(a^{(1)}(y)\right)^{1/2}\nabla w(y)=0$.
(Since for any $\xi\in{\mathbb R}^{n\times d}$,  $\left(a^{(1)}(y)\right)^{1/2}\xi=0$
if and only if $a^{(1)}(y)\xi=0$ by the symmetry \eqref{symm} of non-negative $a^{(1)}$.)
This implies that the expressions in \eqref{aksk} vanish, and hence
$ \langle\,F\,,\,w\,\rangle\,=\,0$, i.e. \eqref{solvmain} holds.

Then, by Theorem \ref{thmsolv}, there exists a unique $u_1\in V^\bot$ such that
\eqref{kuka1} holds. Verify that such a $u_1$ satisfies \eqref{weylrep}. We have
\[
\left\|\,\eta(y)\,-\,\left(a^{(1)}(y)\right)^{1/2}\nabla_yu_1(y)\right\|_2^2\,=\,
\left(\eta(y)\,,\,\, \ \eta(y)\,-\,\left(a^{(1)}(y)\right)^{1/2}\nabla_yu_1(y)\right)_2\,-
\]
\begin{equation}
\left(\left(a^{(1)}(y)\right)^{1/2}\nabla_yu_1(y)\,,\,\, \
\eta(y)\,-\,\left(a^{(1)}(y)\right)^{1/2}\nabla_yu_1(y)\right)_2\,
=:S_1\,+\,S_2. \label{dagger}
\end{equation} Now, it follows
from \eqref{kuka1} that
$\psi(y):=\eta(y)\,-\,\left(a^{(1)}(y)\right)^{1/2}\nabla_yu_1(y)\in\,W$
(see \eqref{wspace2}$\,$), and hence, by the assumption
\eqref{ortw} of the theorem, $S_1=0$. On the other hand,
\[
S_2:=\int_Q
\left(a^{(1)}(y)\right)^{1/2}\nabla_yu_1(y)\,\cdot\,\psi(y)\,dy\,=\,
\int_Q
\nabla_yu_1(y)\,\cdot\,\left(a^{(1)}(y)\right)^{1/2}\psi(y)\,dy\,
=\, 0
\]
by \eqref{kuka1}. Hence \eqref{dagger} yields
$\left\|\,\eta(y)\,-\,\left(a^{(1)}(y)\right)^{1/2}\nabla_yu_1(y)\right\|_2=0$
implying \eqref{weylrep}.

The above construction also ensures that $u_1$ is determined  uniquely up to any function from $V$,
in particular is unique in $V^\bot$.
\end{proof}

\begin{remark}
If $n=1$ and $a^{(1)}_{1j1q}\equiv \delta_{jq}$, Theorem \ref{weyldec} recovers a classical Weyl's 
decomposition for vector fields in $\left(L^2_\#(Q)\right)^d$ into the sum of a divergence-fee and of a 
potential fields: any vector field $w\in \left(L_\#^2(Q)\right)^d$ 
is uniquely decomposed into the orthogonal sum $w=\psi+\eta$, where $\psi\in W$ and $\eta\in W^\bot$. Now, 
according to \eqref{wspace} with $n=1$ and $a^{(1)}_{1j1q}\equiv \delta_{jq}$, $\mbox{div}_y\psi=0$ i.e. 
$\psi$ is divergence-free, and by the theorem $\eta=\nabla_y u_1$ for some $u_1$ i.e. $\eta$ is a potential field. 
\end{remark} 

The above listed properties, in particular those in Theorems
\ref{thmsolv} and \ref{weyldec}, allow to pass to the limit
in equation \eqref{pde}, equivalently in its weak form
\eqref{weak}, as we execute in the next section.

\section{The two-scale limit problem.}\label{2slimprob}

We establish first an important property connecting, under the condition \eqref{keyass}, the generalized two-scale limit flux $\xi_0(x,y)$ to the two-scale limit field $u_0(x,y)$, see \eqref{2sc1}--\eqref{2sc3}. 

We introduce the following set of ``product'' test functions in $L^2(\Omega; W)$. Let $\Psi(x,y)=g(x)\psi(y)$, where 
$g\in C_c^\infty(\overline\Omega)$ and $\psi\in W$, see \eqref{wspace}-\eqref{wspace2}. Here $C_c^\infty(\overline\Omega)$ consists of restrictions to $\Omega$ of all the scalar functions 
in $\Omega$ 
from $C_0^\infty({\mathbb R}^d)$, i.e. of infinitely differentiable functions with a compact support in the whole of ${\mathbb R}^d$. 
We note that the linear span of such test functions $\Psi$ is dense in 
$L^2(\Omega; W)$: e.g. an arbitrary $\Psi(x,y)\in L^2(\Omega; W)\subset 
\left(L^2(\Omega\times Q)\right)^{n\times d}$ is approximated 
in $\left(L^2(\Omega\times Q)\right)^{n\times d}$ by 
linear span of $g_m(x)\tilde\psi_m(y)$ with $g_m(x)\in C_0^\infty(\Omega)\subset C_c^\infty(\overline\Omega)$ and 
$\tilde\psi_m(y)\in C^\infty_\#(Q)$, and then setting $\psi_m(y)=P_W\tilde\psi_m(y)$ with $P_W$ denoting 
orthogonal projection on $W$ in $\left(L_\#^2( Q)\right)^{n\times d}$. 

The following important lemma holds. 

 



\begin{lemma}\label{fluxcorr}
Let $u_0(x,y)$ and  $\xi_0(x,y)$ be as in Lemma \ref{lem22},
and let condition \eqref{keyass} hold. Then 
the following integral identity holds: 
 \[
\forall \, \Psi(x,y)=g(x)\psi(y), \ g\in C_c^\infty(\overline\Omega), \psi\in W, \ \ \ \ \ \ 
\int_\Omega\int_Q \xi_0(x,y)\,\cdot\,\Psi(x,y)\,dx\,dy,\,= 
\]
\begin{equation}
-\,\int_\Omega\int_Q
u_0(x,y)\,\cdot\,\mbox{div}_x\left(\,\left(a^{(1)}(y)\right)^{1/2}
\Psi(x,y)\,\right) \,dx\,dy. 
\label{pfl34}
\end{equation}
\end{lemma}
\begin{remark}
Notice that, importantly, in \eqref{pfl34} $\Psi$ is {\it not} in $C_0^\infty(\Omega)$ in $x$ and 
so may adopt non-zero values on the boundary $\partial\Omega$ of 
$\Omega$.  In this respect, 
\eqref{pfl34} encodes in some sense `boundary conditions' for $u_0(x,y)$, $x\in\partial\Omega$, which may remain inherited for degenerate $a^{(1)}$ in the limit 
$\varepsilon \to 0$ from the zero Dirichlet boundary conditions for 
$u^\varepsilon\in \left(H_0^1(\Omega)\right)^n$ in \eqref{pde}.  
\end{remark}
\begin{proof}
Let  $\Psi(x,y)=g(x)\psi(y)$, where $g\in C_c^\infty(\overline\Omega)$, $\psi\in W$, and $W$ is defined by \eqref{wspace2}.
Then, by \eqref{2sc3},
\begin{equation}
\lim_{\varepsilon \to 0} \int_\Omega
\left(a^{(1)}\left(\frac{x}{\varepsilon}\right)\right)^{1/2}\nabla u^\varepsilon(x)
\,\cdot\, \Psi\left(x, \frac{x}{\varepsilon}\right) \,dx\,=\,
\int_\Omega\int_Q \xi_0(x,y)\,\cdot\,\Psi(x,y)\,dx\,dy.
\label{pfl31}
\end{equation}
On the other hand, 
\[
\int_\Omega
\left(a^{(1)}\left(\frac{x}{\varepsilon}\right)\right)^{1/2}\nabla u^\varepsilon(x)
\,\cdot\, \Psi\left(x, \frac{x}{\varepsilon}\right) \,dx\,=\, 
\]
\begin{equation}
\int_\Omega \nabla\biggl(g(x) u^\varepsilon(x)\biggr)\cdot 
\left(a^{(1)}\left(\frac{x}{\varepsilon}\right)\right)^{1/2} \psi\left( \frac{x}{\varepsilon}\right) dx
-\int_\Omega\, u^\varepsilon(x)\cdot\left.\left(\mbox{div}_x\left(\left(a^{(1)}(y)\right)^{1/2}
\Psi(x,y)\right)\right)\right\vert_{y=x/\varepsilon}dx.
\label{pfl320}
\end{equation}
We notice first that, for any fixed $\varepsilon>0$, the first term on the right hand side 
is zero. 
This follows e.g. from extending $g(x)u^\varepsilon(x)$ by zero outside $\Omega$ and then applying partition of unity arguments and using \eqref{wspace2}. 
Hence \eqref{pfl320} gives 
\begin{equation} 
\int_\Omega
\left(a^{(1)}\left(\frac{x}{\varepsilon}\right)\right)^{1/2}\nabla u^\varepsilon(x)
\cdot \Psi\left(x, \frac{x}{\varepsilon}\right) dx=
-\int_\Omega\, u^\varepsilon(x)\cdot\left.\left(\mbox{div}_x\left(\left(a^{(1)}(y)\right)^{1/2}
\Psi(x,y)\right)\right)\right\vert_{y=x/\varepsilon}dx, 
\label{pfl32}
\end{equation} 
and passing to the limit as $\varepsilon \to 0$ and using \eqref{2sc1} then yields
\[
\lim_{\varepsilon \to 0} \int_\Omega
\left(a^{(1)}\left(\frac{x}{\varepsilon}\right)\right)^{1/2}\nabla u^\varepsilon(x)
\,\cdot\, \Psi\left(x, \frac{x}{\varepsilon}\right) \,dx\,=\,
\]
\begin{equation}
-\,\,\int_\Omega\int_Q
u_0(x,y)\,\cdot\,\mbox{div}_x\left(\,\left(a^{(1)}(y)\right)^{1/2}
\Psi(x,y)\,\right) \,dx\,dy. \label{pfl33}
\end{equation}
Comparing \eqref{pfl31} and \eqref{pfl33} results in identity \eqref{pfl34}. 
\end{proof} 
\vspace{.1in}

The identity \eqref{pfl34} encodes the relation between the generalized limit flux $\xi_0(x,y)$ and 
the limit field $u_0(x,y)$. 
Motivated by \eqref{pfl34}, we introduce the following linear subspace $U$ of Hilbert space $L^2(\Omega;V)$: 
\begin{eqnarray} 
 U\,:=\,\Biggl\{\, u(x,y)\in L^2\left(\Omega;\,V\right)\,\Biggr\vert \,
\exists\, \xi(x,y)\in L^2(\Omega; W)\,\,\mbox{such that, }\,\, \nonumber\\ 
\forall \, \Psi(x,y)=g(x)\psi(y), \ g\in C_c^\infty(\overline\Omega), \psi\in W, \ \ \ \ \ \ 
\int_\Omega\int_Q \xi(x,y)\,\cdot\,\Psi(x,y)\,dx\,dy\,= \nonumber\\ 
\left. 
-\,\int_\Omega\int_Q
u(x,y)\,\cdot\,\mbox{div}_x\left(\,\left(a^{(1)}(y)\right)^{1/2}
\Psi(x,y)\,\right) dxdy  \ \ 
\right\}. \ \ \ \ \ \ 
\label{pfl341}
\end{eqnarray}
Obviously, by Lemma \ref{fluxcorr}, $u_0(x,y)\in U$. We will see that $U$ forms a domain for the sesquilinear 
form for the two-scale limit operator, and so can be viewed as a two-scale generalization of $H_0^1(\Omega)$ 
in the classical homogenization. 

For any $u(x,y)\in U$, the associated $\xi(x,y)$ in \eqref{pfl341} is found uniquely due to the density of (linear span of) $\Psi(x,y)=g(x)\psi(y)$ 
in $L^2(\Omega; W)$. Let $T$ denote the corresponding linear operator, $T:U \rightarrow L^2(\Omega; W)$, $\,Tu:=\xi$. 
 Denote by $P_W$ the orthogonal projector on $W$ with respect to the standard $L^2$ inner product \eqref{ortw}. 
Then, bearing in mind the definition  of $T$ and formally for a moment 
integrating by parts in  \eqref{pfl341}, it can be symbolically written as 
\begin{equation}
\xi(x,y)\,=\,Tu(x,y)\,=:\,P_W\left[\left(a^{(1)}(y)\right)^{1/2}\nabla_x u(x,y)\,\right]\,\in\,L^2(\Omega;\, W). 
\label{toper}
\end{equation} 
We emphasize that the writing in \eqref{toper} is in general formal: for $u\in U$, $\left(a^{(1)}(y)\right)^{1/2}\nabla_x u(x,y)$ is {\it not} generally in 
$\left(L^2(\Omega\times Q)\right)^{n\times d}=L^2\left(\Omega; \left(L^2(Q)\right)^{n\times d}\right)$.

Introduce now another set of `product' test functions in $U$, smooth in $x$: $\phi_0(x,y)=\eta(x)v(y)$ so that $\eta\in C_0^\infty(\Omega)$ and $v\in V$, see \eqref{herset}. 
It is easy to see that $\phi_0(x,y)\,\in\,U$, and 
the corresponding $T\phi_0\in L^2 
(\Omega;W)$ 
is determined, via integration by parts in \eqref{pfl341}, by \eqref{toper} now in the pointwise sense in $x$. 
Further, the following ``corrector'' property holds:  
\begin{proposition}\label{pwcorr}
Let $\phi_0(x,y)= 
\eta(x)v(y)$, where $\eta\in C_0^\infty(\Omega)$ and $v\in V$. 
Then $\phi_0\in U$, and 
there exists a unique ``corrector'' $\phi_1(x,y)\in\,
L^2\left(\Omega; V^\bot\right)$ such that 
\begin{equation}
T\phi_0(x,y)\,
=\,P_W\left[\left(a^{(1)}(y)\right)^{1/2}\nabla_x \phi_0(x,y)\,\right]
=\,\left(a^{(1)}(y)\right)^{1/2}\biggl[\,\nabla_x\phi_0(x,y)\,+\,\nabla_y\phi_1(x,y)\,\biggr].
\label{xi0form}
\end{equation}
Here, for 
all $x\in\Omega$, $\phi_1(x,y)\in\,
V^\bot$ is a unique solution of the corrector problem 
\begin{equation}
\mbox{\text{div}}_y\left(
a^{(1)}(y)\biggl[\,\nabla_x\phi_0(x,y)\,+\,\nabla_y\phi_1(x,y)\,\biggr] \right)\,=\,0,
\label{corrprob}
\end{equation} 
equivalently, 
\begin{equation}
\int_Q 
a^{(1)}(y)\biggl[\,\nabla_x\phi_0(x,y)\,+\,\nabla_y\phi_1(
y)\,\biggr] \cdot 
\nabla_y\psi(y)dy\,=\,0, \ \ \ \forall \psi\in 
\left(H^1_\#(Q)\right)^n. 
\label{corrprob2}
\end{equation} 
\end{proposition}
\begin{proof}
Let $\phi_0(x,y)= \eta(x)v(y)$, $\eta\in C_0^\infty(\Omega)$ and $v\in V$.  
For every fixed $x\in\Omega$, consider the problem \eqref{corrprob2}. 
It follows from Theorem \ref{thmsolv} with 
$\langle F,w\rangle=\,-\,\int_Q a^{(1)}(y)  \nabla_x\phi_0(x,y) \cdot \nabla_yw(y)dy$ 
that \eqref{solvmain} holds and hence  
 \eqref{corrprob} has a unique solution 
$\phi_1(x,\cdot)\in V^\bot$. 
Denoting 
\begin{equation}
\xi(x,y):= \,\left(a^{(1)}(y)\right)^{1/2}\left[ \nabla_x\phi_0(x,y)\,+\,\nabla_y\phi_1(x,y)\,\right], 
\label{xismooth}
\end{equation} 
we notice that $\xi(x,\cdot)\in W$, $\forall x$, by \eqref{corrprob2}, cf. \eqref{wspace}--\eqref{wspace2}, 
and noticing the smooth dependence on $x$,  $\xi(x,y)\in L^2(\Omega;W)$. 
Then the identity in \eqref{pfl341} for $u=\phi_0$ follows from integration by parts in $x$, \eqref{xismooth}, 
the fact that $\Psi(x,\cdot)\in W$, and \eqref{wspace}-\eqref{wspace2}. 

This implies $\phi_0\in U$, and $T\phi_0=\xi$, yielding \eqref{xi0form}.  
\end{proof}
\vspace{.15in} 

One can now pass to the limit in the weak form \eqref{weak} of the original equation as follows. 
Let $f^\varepsilon \stackrel{2}\rightharpoonup f_0(x,y)\in L^2\left(\Omega; \left(L^2(Q)\right)^n\right)$.  
We take as a test function in \eqref{weak} 
$\phi(x)=\phi^\varepsilon(x)=\,\phi_0\left(x,\frac{x}{\varepsilon}\right)$,   
where $\phi_0(x,y)= \eta(x)v(y)$, $\eta\in C_0^\infty(\Omega)$ and $v\in V$.  
The use of \eqref{aeps}, \eqref{2sc1}--\eqref{2sc3}, \eqref{pfl34}, \eqref{toper}, 
and 
\eqref{xi0form} 
results in the following limit form for \eqref{weak}: 
\begin{eqnarray}
\int_\Omega\int_Q& \biggl\{\,
Tu_0(x,y)\,\cdot\,T\phi_0(x,y)\,\,\,\, \ 
\,+\,
\nonumber\\
\biggl.&a^{(0)}(y)\nabla_yu_0(x,y)\cdot\nabla_y\phi_0(x,y)\,+\,
\lambda\,\rho(y)u_0(x,y)\cdot\phi_0(x,y)\,\biggr\}dy\,dx\,=\,\nonumber\\
&\int_\Omega\int_Q\,\rho(y)f_0(x,y)\cdot\phi_0(x,y)\,dy\,dx, 
\ \ \ \ \forall \phi_0(x,y)=\eta(x)v(y), \ \eta\in C_0^\infty(\Omega), \, v\in V.
\nonumber\\
\label{limweak}
\end{eqnarray}

Integral identity \eqref{limweak} can be viewed as a weak form for the limit problem for $u_0(x,y)\,\in\,U$. 
To argue that this is a well-posed problem we first introduce the following sesquilinear quadratic form on $U$: 
\begin{eqnarray}
Q(u,w)\,:=\,
\int_\Omega\int_Q& \biggl\{\,
Tu(x,y) \,
\cdot
\, \overline{Tw(x,y)} 
\,+\, & \, \ \ \ \ 
\nonumber\\
\biggl.&a^{(0)}(y)\nabla_yu(x,y)\cdot\overline{\nabla_yw(x,y)}\,+\, &
\rho(y)u(x,y)\cdot\overline{ w(x,y)}\,\biggr\}dy\,dx. 
\nonumber\\ 
\label{limform}
\end{eqnarray}
The form $Q$ defines an inner product on $U$. 
\begin{lemma}\label{qclosed}
Form $Q$ is closed on $U$. Hence $U$ is a Hilbert space with inner product 
\begin{equation}
\langle u,\,w\rangle_U\,:=\,Q(u,\,w). 
\label{Uinner}
\end{equation}
\end{lemma}  
\begin{proof}
Let $u_j$, $j=1,2,...$, be a Cauchy sequence in $U$, i.e. $\|u_j-u_k\|_U\to 0$ as $j,k\to\infty$, 
where 
\begin{equation}
\|\,u\,\|^2_U\,:=\,Q(u,u).
\label{qUnorm} 
\end{equation} 
 Let $\xi_j:=Tu_j$. Then, according to \eqref{limform}, \eqref{rhobd}, \eqref{alinfty} and 
\eqref{coerc1}\footnote{Notice that assumption \eqref{coerc1} implies similar inequality in 
$\left(H^1_\#(Q)\right)^n$ with the integral over ${\mathbb R}^d$ replaced by integral over $Q$, 
$u\cdot u(y)$ added to the integrand on the left, and the norm on the right replaced by the 
$H^1_\#(Q)$-norm, with some constant $C>0$. This can be seen by e.g. multiplying an infinitely 
periodic $u\in \left(H^1_\#(Q)\right)^n$ by a smooth cut-off function $\chi_R(y)$ such that $\chi_R=1$ 
for $|y|\leq R$, $\chi_R=0$ for $|y|>R+1$ and $|\nabla\chi_R(y)|\leq C$, and taking $R$ large enough.},  
$\|u_j-u_k\|_{L^2(\Omega;V)}\to 0$ and $\|\xi_j-\xi_k\|_{L^2(\Omega;W)}\to 0$. 
Since both $L^2(\Omega;V)$ and $L^2(\Omega;W)$ are complete, there exist $\tilde u\in L^2(\Omega;V)$ 
and $\tilde\xi\in L^2(\Omega;W)$ such that, respectively, $u_j\to \tilde u$ in $ L^2(\Omega;V)$ and 
 $\xi_j\to \tilde\xi$ in $ L^2(\Omega;W)$. 
Taking then arbitrary $\Psi(x,y)=g(x)\psi(y), \ g\in C_c^\infty(\overline\Omega), \psi\in W$,  
one passes to the limit as $j\to \infty$ in 
both the left hand side and the right hand side of \eqref{pfl341} (held with $\xi$ and $u$ replaces by 
$\xi_j$ and $u_j$, respectively, by the definition of $u_j\in U$ and of 
$\xi_j=Tu_j$). Hence \eqref{pfl341} also holds for 
$\tilde\xi$ and $\tilde u$, and therefore $\tilde u\in U$, $\tilde\xi=T\tilde u$, and  $\|u_j - \tilde u\|_U\to 0$, 
which completes the proof. 
\end{proof}

We ultimately need to show that the linear span of the set of test functions $\phi_0$ adopted in \eqref{limweak} is dense 
in $U$ with respect to the norm \eqref{qUnorm}. With this aim, we first introduce a wider set of `smooth compactly supported in $x$' trial fields $\phi_0(x,y)\in U$ 
for which \eqref{limweak} still holds. 

\begin{definition}
Consider all $\phi(x,y)\in L_c^2(\Omega;V)$ of functions from $L^2(\Omega;V)$ with a compact (in $x$) 
support $\mbox{supp}_x \phi$ in $\Omega$. Let a scalar function $\zeta\in C_0^\infty\left({\mathbb R}^d\right)$ be such 
that $\mbox{diam}(\mbox{supp}\, \zeta)< \mbox{dist} (\mbox{supp}_x \phi, \partial\Omega)$, and consider an ``$x$-smoothed'' function 
\begin{equation}
\phi_0(x,y)\,=\, \int_{{\mathbb R}^d} \zeta(x-x')\phi(x',y) dx'  
\label{c0infv}
\end{equation}
(which is still in $L^2(\Omega;V)$ and has a compact support in $x$). 
We denote by $\tilde C_0^\infty(\Omega;V)$ the linear span of all such functions $\phi_0(x,y)$. 
\end{definition} 

\begin{lemma}\label{lemmactild}
(i) All $\phi_0\in \tilde C_0^\infty(\Omega;V)$ belong to $U$, with 
\begin{equation}
T\phi_0(x,y)\,=\,\int_{{\mathbb R}^d} T\left(\zeta(x-x')\phi(x',y)\right) dx'. 
\label{Tc0inf}
\end{equation} 
(ii) The identity \eqref{limweak} holds for the extended set of trial fields $\phi_0\in \tilde C_0^\infty(\Omega;V)$.
\end{lemma} 
\begin{proof}

(i) Let $\phi_0\in \tilde C_0^\infty(\Omega;V)$ be associated with $\phi(x,y)\in L_c^2(\Omega;V)$ via \eqref{c0infv}. 
Then, regarding $x'\in\Omega$ as a parameter, for almost every $x'\in\Omega$, $\phi_0(x,y;x')=\zeta(x-x')\phi(x',y)$ is of the product form as in Proposition \ref{pwcorr}: 
$\phi_0(x,y;x')=\eta(x)v(y)$ where for $x'\in \mbox{supp}_x\phi$, $\eta(x)=\zeta(x-x')$ and $v(y)=\phi(x',y)$, and $\eta\equiv v\equiv 0$ 
otherwise. Hence, by Proposition \ref{pwcorr}, \eqref{pfl341} holds for a.e. $x'$ with $u(x,y)=\phi_0(x,y;x')$ and some 
$\xi(x,y;x')=T\left(\zeta(x-x')\phi(x',y)\right)$. 
Integration of \eqref{pfl341} with respect to the parameter $x'$ then yields $\phi_0\in U$, with $T\phi_0(x,y)$ given by \eqref{Tc0inf}. 

(ii) Similarly, \eqref{limweak} holds for almost every $x'\in\Omega$ with $\phi_0(x,y)=\phi_0(x,y;x')$ and $T\phi_0(x,y)=\xi(x,y;x')$ 
constructed above. Hence, integrating \eqref{limweak} in $x'$, the result follows. 
\end{proof}

The following key property will be proved for star-shaped bounded domains $\Omega$, 
although  it can similarly be shown to be valid for rather general domains (see Remark \ref{remepigr} below). 

\begin{definition}
We call a domain $\Omega$ strictly star-shaped (with respect to the origin $x=0\in\Omega$) if, for all small $\delta>0$, 
$\mbox{dist}\left( (1-\delta)\Omega, \,\partial\Omega\right)>0$. 
\end{definition} 

\begin{theorem}\label{cinfdense}
Let $\Omega$ be a  strictly star-shaped bounded domain. 
Then $U$ is the closure of $\tilde C_0^\infty(\Omega; V)$ 
in the norm $\|\cdot\|_U$, see \eqref{qUnorm}. 
\end{theorem} 
\begin{proof}
1. Let $\Omega$ be a bounded domain, strictly star-shaped with respect to origin $O$. Fix $u(x,y)\in U$, let $\xi(x,y)=Tu(x,y)\in L^2(\Omega;W)$, 
and regard both $u$ and $\xi$ as  
functions on the whole  ${\mathbb R}^d$ in $x$ by extending them outside $\Omega$ by zero. 
We aim at constructing a sequence $u_\delta\in \tilde C_0^\infty(\Omega; V)$ such that 
$u_\delta \to u$ in $U$ as $\delta \to 0$. 

To this end, 
for any small $\delta>0$, let $\Omega_\delta:=(1-\delta)\Omega$ and denote 
$d(\delta):=\mbox{dist}\left(\Omega_\delta, \partial\Omega\right)>0$. 
Let $\hat u_\delta(x,y):=u(x/(1-\delta), y)$. Obviously, $\hat u_\delta\in L^2(\Omega,V)$ and the support of $\hat u_\delta$ is contained in 
$\overline{\Omega_\delta}\subset\Omega$. 
Select $\epsilon(\delta)=d(\delta)/2>0$ and let $\zeta_\epsilon(x)$ be a standard mollifying function: 
$\zeta_\epsilon(x)=\epsilon^{-d}\zeta(x/\epsilon)$, where $\zeta(z)\in C^\infty_0({\mathbb R}^d)$, 
$\zeta(-z)=\zeta(z)$, 
$\mbox{supp }\zeta(z) \subset B(0,1)$ and $\int_{{\mathbb R}^d}\zeta(z)dz=1$. 
Consider the $x$-smoothed function 
\[
u_\delta(x,y)\,:=\,\zeta_\epsilon*\hat u_\delta(x,y)\,:=\,\int_{{\mathbb R}^d}\zeta_\epsilon(x-x')\hat u_\delta(x',y)dx'. 
\]
Obviously, by the construction and Lemma \ref{lemmactild}, $u_\delta(x,y)\in \tilde C^\infty_0(\Omega; V)\,\subset\, U$. 

We argue that $u_\delta \to u$ in $U$ as $\delta \to 0$. 
According to \eqref{qUnorm}, \eqref{limform} and \eqref{toper} it suffices to show that $u_\delta \to u$ in 
$L^2(\Omega; V)$ and $Tu_\delta \to Tu$ in $L^2(\Omega; W)$. 

The former assertion immediately follows from the fact that $\hat u_\delta \to u$ in $L^2(\Omega; V)$, 
cf. e.g. \cite{evans}, and from $\left\|u_\delta-\hat u_\delta\right\|_{L^2(\Omega; V)} \to 0$ (trivially 
established via e.g. changing variables $\hat x=x/(1-\delta)$, noticing that $\epsilon \to 0$ as 
$\delta \to 0$ and using the properties of the mollifications, cf. e.g. \cite{evans}).  

2. To prove that $Tu_\delta \to Tu$, 
choose $\Psi(x,y)=\eta(x)\psi(y)$, with $\eta\in C_c^\infty(\overline\Omega)$ and $\psi\in W$, cf. \eqref{pfl341}. 
Then, for the right hand side of \eqref{pfl341} with 
$u$ replaced by $u_\delta \in U$, 
\[
I(\delta)\,:=\,-\,\int_\Omega\int_Q
u_\delta(x,y)\,\cdot\,\mbox{div}_x\left(\,\left(a^{(1)}(y)\right)^{1/2}
\Psi(x,y)\,\right) \,dy\,dx\,=\, 
\]
\[
-\,\int_{\Omega
}\int_Q
\left[\int_{\Omega_\delta}\zeta_\epsilon(x-x')\hat u_\delta(x',y)dx'\right] 
\,\cdot\,\mbox{div}_x\left(\,\left(a^{(1)}(y)\right)^{1/2}
\Psi(x,y)\,\right) \,dy\,dx\,=\,
\]
\begin{equation} 
-\,\int_{\Omega_{\delta}}\int_Q
\hat u_\delta(x',y)\,\cdot\,
I_\epsilon(x',y)\,dydx',  
\label{idel1}
\end{equation} 
where 
\[
I_\epsilon(x',y)\,:=\,\int_{\Omega}\zeta_\epsilon(x-x') 
\mbox{div}_x\left(\,\left(a^{(1)}(y)\right)^{1/2}
\Psi(x,y)\,\right) \,dx, 
\]
having interchanged above the orders of integration. 
Notice that for $x'\in\Omega_\delta$  
the integrand in  $I_\epsilon(x',y)$  is smooth and compactly supported 
in $\Omega$ in $x$. Hence, via integration by parts and straightforward manipulation, 
\begin{equation}
I_\epsilon(x',y)\,=\,
\mbox{div}_{x'}\left(\,\left(a^{(1)}(y)\right)^{1/2}
\hat\Psi_\delta(x',y)\,\right), 
\label{iepsdiv}
\end{equation} 
where 
\begin{equation}
\hat \Psi_\delta(x',y)\,:=\,\left(\zeta_\epsilon*\Psi\right)(x',y):= 
\int_{\Omega}\zeta_\epsilon(x^{''}-x') \Psi(x^{''},y)dx^{''}\,=\hat\eta_\delta(x')\psi(y), 
\label{psihdel}
\end{equation} 
with $\hat\eta_\delta=\zeta_\ep*\eta\in C_c^\infty(\overline\Omega)$. 
Changing in \eqref{idel1}--\eqref{iepsdiv} the integration variable ($x=x'/(1-\delta)$), 
and introducing 
$\Psi_\delta(x,y):= (1-\delta)^{d-1} 
\hat \Psi_\delta\left((1-\delta)x,y\right)\,=\eta_\delta(x)\psi(y)$, 
$\eta_\delta(x):=(1-\delta)^{d-1}\hat\eta_\delta((1-\delta)x)\in C_c^\infty(\overline\Omega)$,      
results in 
\[
I(\delta)\,=\,-\,\int_\Omega\int_Q
u(x,y)\,\cdot\,\mbox{div}_x\left(\,\left(a^{(1)}(y)\right)^{1/2}
\Psi_\delta(x,y)\,\right) \,dy\,dx,  
\]
which reproduces the right hand side of \eqref{pfl341} for $\Psi$ replaced by $\Psi_\delta$. 
Hence, applying  
\eqref{pfl341} to $u\in U$ and  $\Psi_\delta=\eta_\delta(x)\psi(y)$ 
(recalling $\eta_\delta\in C^\infty_c(\overline\Omega)$ and $\eta\in W$), 
results in 
\begin{equation} 
I(\delta)\,=\,\int_\Omega\int_Q \xi(x,y)\,\cdot\,\Psi_\delta(x,y)\,dy\,dx\,=\, 
\int_\Omega\int_Q \xi_\delta(x,y)\,\cdot\,\Psi(x,y)\,dy\,dx,  
\label{idel2}
\end{equation}
where, via \eqref{psihdel}, and a further change of integration variables, 
\begin{equation}
\xi_\delta(x,y)\,:=\,
(1-\delta)^{-1}
\int_{\Omega_\delta}\zeta_\epsilon(x-x')\xi\left(x'/(1-\delta),y\right)dx'.
\label{xidel}
\end{equation} 
By the uniqueness of $\xi$ in \eqref{pfl341} for $u$ replaced by $u_\delta\in U$, 
\eqref{idel2} yields 
$Tu_\delta=\xi_\delta$. 
It is now straightforward to check for $\xi_\delta$,  as given by \eqref{xidel}, 
that $\xi_\delta\to \xi=Tu$ in $L^2(\Omega; W)$ as $\delta\to 0$. 
Therefore $Tu_\delta \to Tu$ in $L^2(\Omega; W)$ as $\delta\to 0$, which completes the 
proof. 
\end{proof}


\begin{remark}\label{remepigr}

Since all the arguments in the above proof have been {\it local in} $x$, using a suitable partition of 
unity in $x$ the proof can be extended to e.g. any domains which can be 
presented locally as either strictly star-shaped domains or (locally) 
epigraphs of arbitrary continuous functions. 

Indeed, given $u(x,y)\in U$ with associated $\xi(x,y)=Tu(x,y)\in L^2(\Omega;W)$  and 
$\chi(x)\in C_0^\infty({\mathbb R}^d)$, one can see that $\chi(x)u(x,y)\in U$ with associated 
$T(\chi(x)u(x,y))=\chi(x)\xi(x,y)+
P_W\left[\left(a^{(1)}(y)\right)^{1/2}u(x,y)\otimes\nabla\chi(x)\right]$, as found from \eqref{pfl341} by integration by parts. 

This observation allows employing 
the partition of unity. For local epigraphs of continuous functions, $x_d>f(x_1,..,x_{d-1})$, the above 
proof modifies in an obvious way by replacing the $(1-\delta)$-contractions of the star-shaped domain by 
simple $\delta$-translations in the positive $x_d$-direction. 
The proof in the case of 
$\Omega={\mathbb R}^d$ can be done similarly to the above with $\hat u_\delta$ 
replaced by multiplying $u$ by a suitable family of cut-off functions, which can be combined with the above 
partition of unity arguments for extending the result to arbitrary bounded or unbounded domains with locally 
strictly star-shaped or `local-epigraph' boundaries. 
The routine details are omitted. 
\end{remark}

Lemmas \ref{lemmactild}(ii) and \ref{cinfdense} imply that the identity \eqref{limweak}, which can be 
rewritten via \eqref{limform} as $Q(u_0,\phi_0)+(\lambda-1)(u_0,\phi_0)_{H}=(f_0,\phi_0)_{H}$, where 
 $H:=L^2\left(\Omega; \left(L^2_\rho (Q)\right)^n\right)$ is Hilbert space with inner product
\beq
(u_1,u_2)_H\,=\,\int_{\Omega\times Q} \rho(y)u_1(x,y)\cdot\overline{u_2(x,y)}\, dx dy, 
\label{innerh}
\eeq
 holds for all 
$\phi_0\in U$. Further, the proof of Lemma \ref{qclosed} implies that, for any $\lambda>0$, the sesquilinear form determined by 
the left hand side of \eqref{limweak} is bounded and coercive in the Hilbert space $U$, on which the 
right hand side of \eqref{limweak} specifies a linear continuous functional on $U$. This implies by the 
Lax-Milgram lemma that, for any $f_0\in H$ 
\eqref{limweak} 
has a unique solution $u_0(x,y)\in U$. The latter uniqueness in turn implies that the solutions $u^\ep(x)$ 
of the original problem \eqref{weak} weakly two-scale converge to $u_0(x,y)$, without the need for 
extracting a subsequence. 
These are key technical results of this work, with numerous implications, so we summarize that below 
as following theorem: 

\begin{theorem}\label{maintheor}
Let the assumptions \eqref{rhobd}--\eqref{coerc1}, as well as the key assumption \eqref{keyass}, hold. 
Then, for $\Omega$ from any of the above described classes, for any $\lambda>0$ and for any $f^\varepsilon \stackrel{2}\rightharpoonup f_0(x,y)\in L^2\left(\Omega; \left(L^2(Q)\right)^n\right)$, the unique solutions $u^\ep$ of \eqref{weak} 
weakly two-scale converge to $u_0(x,y)\in U$, $u^\varepsilon(x)\,\stackrel{2}\rightharpoonup\, u_0(x,y)$, which 
uniquely  satisfies the integral identity \eqref{limweak} for all $\phi_0\in U$. The associated generalized 
fluxes $\left(a^{(1)}(x/\eps)\right)^{1/2}\nabla u^\ep(x)$ weakly two-scale converge to $\xi_0(x,y)=Tu_0(x,y)$ 
as defined by \eqref{pfl34}. 
\end{theorem} 

As we show below, 
this will further imply the weak and strong (pseudo-)resolvent 
convergences of the operators, with further implications for convergence of related semigroups and 
associated time-dependent Cauchy problems, and for certain spectral convergence.

\section{The two-scale limit operator and the resolvent convergence}\label{2sclop}

\subsection{The limit operator} 

The above construction defines, in a standard way, a self-adjoint two-scale limit operator $A_0$ in Hilbert space $H_0$ defined as the closure of $U$ in 
the Hilbert space $H=L^2\left(\Omega; \left(L^2_\rho (Q)\right)^n\right)$. 
Indeed, due to Lemma \ref{qclosed}, the non-negative symmetric sesquilinear form $Q(u,w)$ given on $U\times U$ by \eqref{limform} is 
closed and densely defined in $H_0$. Hence it defines a self-adjoint operator $A_0$ in $H_0$ 
with a dense domain $D(A_0)\subset U$: 
\beq
D(A_0)=\left\{u(x,y)\in U\,:\, \exists \mbox{ (unique)}\, w=:A_0u\in H_0\,\mbox{such that}\, \beta(u,v)=(w,v)_H,\, \ \forall v\in U \right\},  
\label{da0}
\eeq
where, cf. \eqref{limform}, 
\beq
\beta(u,v):=\int_\Omega\int_Q\left\{ 
T u(x,y)\,
\cdot 
\overline{T v(x,y)}\,
\,+\,a^{(0)}(y)\nabla_yu(x,y)\cdot\overline{\nabla_y v(x,y)}\,\right\} \,dx\,dy. 
\label{beta}
\eeq 

This fully determines $A_0$ in the general case under the key assumption \eqref{keyass} and $\Omega$ as 
in Remark \ref{remepigr}. The above general 
description of the limit operator $A_0$ may need to be specialized to be made more explicit for particular examples: see e.g. 
\cite{cooperPhD, Cooper14, CKS14, CC15} where such a specialization was performed for some of the examples in 
Section \ref{keyassexamples}. 

Loosely, e.g. assuming sufficient regularity of $u(x,y)$ as well as of $a^{(1)}(y)$, $a^{(0)}(y)$ and 
$\rho(y)$ 
or in an 
appropriate distribution sense, $A_0u$ may be interpreted as follows. As, cf. \eqref{da0}, for $u\in D(A_0)$, 
$A_0u\in H_0$ with $(A_0u,v)_H=\beta(u,v)$ for all $v\in U$, from \eqref{innerh} and \eqref{beta}, 
\[
\int_\Omega\int_Q A_0u(x,y)\cdot \rho(y)\overline{v(x,y)}\,dx\,dy\,=\,
\]
\[
\ \ \ \ \ \ = 
\int_\Omega\int_Q\left\{ 
T u(x,y)\,
\cdot 
\overline{T v(x,y)}\,
\,+\,a^{(0)}(y)\nabla_yu(x,y)\cdot\overline{\nabla_y v(x,y)}\,\right\} \,dx\,dy.
\]
Therefore, formally integrating by parts, 
\[
(A_0u)(x,y)\,=\,
P\left[T^*Tu\,-\,\rho^{-1}(y)\mbox{div}_y \left( a^{(0)}(y)\nabla_yu\,\right)\,\right], 
\]
where $T^*:L^2(\Omega;W)\to H_0$ is the adjoint of $T$ and $P$ is the orthogonal projector from $H$ to $H_0$ (with respect to the 
inner product \eqref{innerh}). Further, for regular enough functions, $T$ can be represented via 
\eqref{xi0form}, and from \eqref{pfl341}, 
\[
(T^*\Psi)(x,y)\,=\,-\,P\left[\rho^{-1}(y)
\mbox{div}_x\left(\,\left(a^{(1)}(y)\right)^{1/2}
\Psi(x,y)\,\right) dxdy
\right]. 
\]
As a result, we arrive at the following (formal) representation for the two-scale limit operator $A_0$: 
\[
(A_0u)(x,y)\,=\,
-\,P\left[\,\rho^{-1}(y)\mbox{div}_x\left(\left(a^{(1)}(y)\right)^{1/2}
P_W \left[\left(a^{(1)}(y)\right)^{1/2}\nabla_xu(x,y)\right]\right)\,\,+\right.
\]
\[
\left.
\ \ \ \ \ \ \ \ \ \ \ \ \ 
\rho^{-1}(y)\mbox{div}_y \left( a^{(0)}(y)\nabla_yu\right)\right].  
\]
Here $P_W$ is the standard $L^2$-orthogonal projector on the space $W$ of admissible micro-fluxes, see \eqref{wspace}, 
which can be constructed, cf \eqref{xi0form}, via solving the `generalized' corrector problem: 
\[
\mbox{div}_y\left(\,a^{(1)}(y)\left[ \nabla_xu(x,y)+\nabla_yu_1(x,y)\right]\right)\,=\,0.  
\]

\subsection{The weak two-scale resolvent convergence} 

Recall that for $u^\ep$ solving the original problem \eqref{weak}, equivalently \eqref{pde}, it can be 
written as $u^\ep=\left(A_\ep+\lambda I\right)^{-1}f^\ep$, see \eqref{resolveps},  where $A_\ep$ is non-negative 
self-adjoint operator in Hilbert space $H_\ep=\left(L^2(\Omega)\right)^n$ equipped with inner product 
$(u,v)_{H_\ep}:=\int_\Omega u(x)\cdot\rho^\ep(x)v(x)dx$. Further, the limit weak formulation \eqref{limweak} 
 is equivalently recast via \eqref{beta} and \eqref{innerh} as 
\[
\beta(u_0,\phi_0)\,+\,\lambda (u_0,\phi_0)_H\,=\,(f_0,\phi_0)_H, \ \ \ \mbox{for all } \phi_0\in U,  
\]
for a given $f_0\in H$. This immediately implies, cf. \eqref{da0}, that for the unique solution $u_0$ 
of \eqref{limweak}, 
$u_0\in D(A_0)$ and $A_0u_0+\lambda u_0=Pf_0$ where $P$ is the above introduced orthogonal projector from $H$ 
on $H_0$. Therefore, $u_0=(A_0+\lambda I)^{-1}Pf_0$ and so 
Theorem \ref{maintheor} can be immediately re-stated as follows.  

\begin{corollary}\label{w2src2}
Under the assumptions \eqref{rhobd}--\eqref{coerc1} and \eqref{keyass} and for any $\Omega$ as in Remark 
\ref{remepigr}, 
let $f^\varepsilon \stackrel{2}\rightharpoonup f_0(x,y)\in H=L^2\left(\Omega; \left(L^2(Q)\right)^n\right)$. 
Then, for all $\lambda>0$, 
\beq
u^\ep=\left(A_\ep+\lambda I\right)^{-1}f^\ep\,\,\stackrel{2}\rightharpoonup \,\,\left(A_0+\lambda I\right)^{-1}Pf_0, 
\ \ \mbox{as $\ep\to 0$}. 
\label{w2src3}
\eeq
The associated generalized 
fluxes $\left(a^{(1)}(x/\eps)\right)^{1/2}\nabla u^\ep(x)$ weakly two-scale converge to $\xi_0(x,y)=Tu_0(x,y)$ 
as defined by \eqref{pfl34}. 
\end{corollary} 
The corollary can be interpreted as a weak two-scale (pseudo-)resolvent convergence, see e.g. \cite{Zhikov2000,
ZhP07,Past05}: the resolvents acting on weakly two-scale convergent sequences weakly two-scale converge to the 
resolvent of the limit operator acting in the orthogonal projection of $f_0$ on $H_0$. 

The latter has further important implications as discussed in the next section. 

\section{Implications of the weak resolvent convergence}\label{w2scimpl} 

Corollary \ref{w2src2} has a number of implications, valid under some abstract assumptions, see \cite{ZhP07,Past05} which include our general case. We state some of these implications below, providing some brief comments. 
We notice first that taking the weak and strong two-scale convergences defined in \eqref{w2scdef} and 
\eqref{s2scdef} as abstract weak and strong convergences of elements of $H_\ep$ to elements of $H$, is easily 
checked to be consistent with the  Definition 1.1 and assumption (1.1) of \cite{Past05}. 

In the rest of this section we assume that \eqref{rhobd}--\eqref{coerc1} and \eqref{keyass} hold, as well as
that $\Omega$ is as in Remark \ref{remepigr}. 

1. {\it Strong two-scale (pseudo-)resolvent convergence.} 
It can be easily checked directly using \eqref{s2scdef} and the self-adjointness of $(A_\ep+\lambda I)^{-1}$ 
and $(A_0+\lambda I)^{-1}$ in $H_\ep$ and $H_0$ respectively cf. 
\cite{Zhikov2000}, and was also proved in generality in e.g. 
\cite{Past05} Lemmas 2.4 and 2.5, that Corollary \ref{w2src2} implies analogous {\it strong} two-scale 
(pseudo-)resolvent convergence. This has further important implications and we state this as the following theorem. 
\begin{theorem}\label{s2src} 
If 
$f^\varepsilon(x) \stackrel{2}\rightarrow f_0(x,y)\in H=L^2\left(\Omega; \left(L^2(Q)\right)^n\right)$,  
then, for all $\lambda>0$, 
\beq
u^\ep=\left(A_\ep+\lambda I\right)^{-1}f^\ep\,\,\stackrel{2}\rightarrow \,\,\left(A_0+\lambda I\right)^{-1}Pf_0, 
\ \ \mbox{as $\ep\to 0$}. 
\label{s2src2}
\eeq
The associated generalized 
fluxes $\xi^\ep(x)=\left(a^{(1)}(x/\eps)\right)^{1/2}\nabla u^\ep(x)$ also {\it strongly} two-scale converge to $\xi_0(x,y)=Tu_0(x,y)$ 
as defined by \eqref{pfl34}. 
\end{theorem} 
Notice that, additionally to \cite{Past05}, the above theorem also includes the strong two-scale convergence of the generalized fluxes. Trying to be concise here, 
the latter can be inferred by first setting in \eqref{weak} $\phi=u^\ep$ and recalling that its both sides converge, 
as $\ep\to 0$, to \eqref{limweak} with $\phi_0$ replaced by $u_0$. Then we notice that the last terms on the 
left and the right hand sides of \eqref{weak} converge to those of \eqref{limweak} since 
$u^\varepsilon(x) \stackrel{2}\rightarrow u_0(x,y)\in H_0$. Further, for non-negative (cf \eqref{coerc1})  variational 
functional $I_\ep(u):=\int_\Omega 
\ep^2 \left(a^{(0)}\left({x}/{\varepsilon}\right)  + a^{(1)}\left(x/{\varepsilon}\right)\right)  
{\nabla} u(x)\cdot
\nabla u(x) dx$, $u\in \left(H_0^1(\Omega)\right)^n$, 
a two-scale weak lower semicontinuity property holds: if 
$u^\varepsilon(x) \stackrel{2}\rightharpoonup u_0(x,y)\in L^2\left(\Omega;\left(H^1_\#(Q)\right)^n\right)=:V_0$ 
then $\liminf_{\ep\to 0} I_\ep(u^\ep) \geq I_0(u_0)$, where 
$I_0(u_0):=\int_{\Omega \times Q}
\left(a^{(0)}\left(y\right)  + a^{(1)}\left(y\right)\right)  
{\nabla_y} u_0(x,y)\cdot
\nabla_y u(x,y) dx dy$. The latter can be shown by, for example, adjusting the argument of Zhikov in 
\S 2.3 (iii)  of \cite{Zhikov2000} to $0\leq I_\ep(u_\ep(x)-\Phi_k(x,x/\ep))$ with $\Phi_k(x,y)$ a linear 
combination of $\phi_i(x)b_i(y)$, and then choosing $\Phi_k(x,y)\to u_0(x,y)$ in 
$V_0$ (e.g. choosing as $\Phi_k$ the truncated Fourier series of $u_0(x,y)$ in $Q$-periodic $y$). 
Together with \eqref{weaknormstrong} for the first terms in \eqref{weak} and \eqref{limweak}, which are 
respectively $\|\xi^\ep(x)\|^2_2$ and $\|Tu_0(x,y)\|^2_2$ (since 
$\xi^\varepsilon \stackrel{2}\rightharpoonup Tu_0(x,y)$ and hence by \eqref{liminf} a priori 
$\liminf_{\ep\to 0}\|\xi^\ep(x)\|^2_2\geq \|Tu_0(x,y)\|^2_2$), this implies 
$\xi^\varepsilon \stackrel{2}\rightarrow Tu_0(x,y)$ as claimed. Notice that the argument 
also implies that in fact $\lim_{\ep\to 0} I_\ep(u^\ep) = I_0(u_0)$, resulting in turn  
in $\varepsilon \nabla u^\varepsilon  \stackrel{2}\rightarrow \nabla_y u_0(x,y)$, cf. \eqref{2sc2}. 
\vspace{.1in}

2. {\it Partial convergence of spectra.} 
Let $\mbox{Sp} A_\ep$ and  $\mbox{Sp} A_0$ be the spectra of the self-adjoint operators $A_\ep$ and $A_0$, 
respectively. Then, as discussed e.g. in \cite{Zhikov2000} and shown in abstract generality in 
\cite{Past05} Theorem 8.1, the strong two-scale resolvent convergence of the above Theorem \ref{s2src} 
automatically implies a ``part'' of the Hausdorff convergence of the spectra, namely 
\begin{corollary}\label{h1spec}
For any $\mu_0\in \mbox{Sp} A_0$ there exist $\mu_\ep\in \mbox{Sp} A_\ep$ such that $\mu_\ep\to\mu_0$. 
\end{corollary} 
Therefore any point on the spectrum of the limit operator $A_0$ for small enough $\ep$ is approximated 
by points in the spectrum of $A_\ep$. 

The ``converse'' part of the Hausdorff convergence, i.e. that $\mu^\ep\to\mu_0$, $\mu^\ep\in \mbox{Sp} A_\ep$ 
implies $\mu_0\in\mbox{Sp} A_0$ does not generally 
hold, see e.g. \cite{CKS14}\footnote{As clarified e.g. in \cite{CKS14,cooperPhD},  
for $\Omega={\mathbb R}^d$ this 
corresponds to non-vanishing contributions to the limit Floquet-Bloch spectrum as $\ep\to 0$ from 
the quasi-periodicity parameter (quasi-momentum) $\theta\neq 0$, for which the present two-scale description restricted to periodic functions ($\theta=0$) appears insufficient.}. It does hold however in a 
number of important examples, see e.g. \cite{Zhikov2000,Zhikov2005,ZhP13,Cooper14}, which has then to be proved by separate means and sometimes allows to establish the existence of band gaps in the spectrum of $A_\ep$ for small enough $\ep$. 

3. {\it Strong convergence of spectral projectors.} 
As again discussed in e.g. \cite{Zhikov2000}, and then shown in an abstract generality in \cite{Past05} 
Theorem 8.4, the strong two-scale resolvent convergence of Theorem \ref{s2src} implies also the convergence of 
spectral projectors. Denote $E_\ep(\lambda)$ and $E_0(\lambda)$ the spectral projectors of the non-negative 
self-adjoint operators $A_\ep$ and 
$A_0$ respectively, i.e. for their spectral decompositions: 
\beq
A_\ep\,=\,\int_0^\infty \lambda \,\, dE_\ep(\lambda), \ \ \ \ 
A_0\,=\,\int_0^\infty \lambda \,\,  dE_0(\lambda). 
\label{spectheor}
\eeq 
Then 
\begin{corollary}\label{specproj}
If $\lambda$ is not an eigenvalue of $A_0$, then $E_\ep(\lambda)f^\ep(x)\,\stwoscale\,E_0(\lambda)f_0(x,y)$ as long as $f^{\ep}(x)\stwoscale f_0(x,y)\in H_0$. 
\end{corollary}

4. {\it Convergence of semigroups and convergence of Cauchy problems for 
time-dependent initial value problems.} 
As again discussed in \cite{Zhikov2000} and then shown in abstract generality in \cite{ZhP07} and in 
\cite{Past05}, the strong (equivalently weak) two scale (pseudo-)resolvent convergence akin to that in the above 
Theorem \ref{s2src} implies appropriate two-scale convergence of associated semigroups as well as of related 
evolution Cauchy problems with time-independent coefficients. 
The reader is referred to \cite{Past05} for an abstract account of some scenarios for such 
convergences, most of which can be 
specialized to our case. We state below a couple of particular 
results from \cite{Past05}, as adapted and extended to our problem. 

The non-negative self-adjoint operators $A_\ep$ and $A_0$ in the respective Hilbert spaces $H_\ep$ and 
$H_0$ 
generate strongly continuous 
contraction semigroups, denoted $\left(S_\ep(t)\right)_{t\geq 0}=e^{-tA_\ep}$ and 
$\left(S_0(t)\right)_{t\geq 0}=e^{-t A_0}$. 

The following theorem results from e.g. specializing Theorem 1.4 of \cite{Past05}, which is turn a 
modification of Trotter-Kato theorem for variable Banach spaces cf. \cite{ZhP07},  to our setting. 
\begin{theorem}\label{parsem}

(i) The strongly continuous contraction semigroups $S_\ep(t)$ associated with 
$A_\ep$ strongly two-scale 
converge pointwise in $t$ to the semigroup $S_0(t)$ associated with $A_0$, i.e. 
if $f^{\ep}(x)\stwoscale f_0(x,y)\in H_0$ then 
for all $t\geq 0$,  
\[
e^{-A_\ep t} f^\ep(x) \stwoscale e^{-A_0 t} f_0(x,y).
\]
(ii) Hence, given $T>0$, for   parabolic Cauchy problem 
\beq
\rho^\ep(x)\frac{\partial u^\ep}{\partial t}\,-\,\mbox{div}\left( a^\ep(x)\nabla u^\ep\right)\,=\,0, 
\ \ \ u^\ep(x,0)= f^\ep(x)\in \left(L^2(\Omega)\right)^n, 
\label{pareps}
\eeq
if $f^{\ep}(x)\stwoscale f_0(x,y)\in H_0$,  then for the (unique) 
solution 
$u^\ep$,  
$u^\ep(x,t) \stwoscale u_0(x,y,t)$ for all $t\geq 0$. 
Here $u_0(x,y,t)$ 
is the (unique) 
solution of two-scale limit Cauchy 
problem: 
\beq
\frac{\partial u_0}{\partial t}\,+\,A_0u_0\,=\,0, \ \ u_0(x,y,0)= f_0(x,y). 
\label{parlim}
\eeq
\end{theorem}
Solutions of both Cauchy problems \eqref{pareps} and \eqref{parlim}
 can be a priori understood as strong solutions. 
For example, seek $u^\ep\in C([0,T];H_\ep)$ and $u_0\in C([0,T];H_0)$ respectively, 
with the 
above initial conditions and 
such that 
for all $t>0$, $u^\ep\in D(A_\ep)$ and $u_0\in D(A_0)$ and 
have strong derivatives in $t$ 
with values in 
$H_\ep$ and $H_0$, 
and $\frac{du^\ep}{dt}+A_\ep u^\ep=0$ and $\frac{du_0}{dt}+A_0 u_0=0$ for all $t>0$. Then $u^\ep(t)=e^{-tA_\ep}f^\ep$ and $u_0(t)=e^{-tA_0}f_0$ are readily checked to be solutions, and the  
uniqueness follows in a standard way from the non-negativity of $A_\ep$ and $A_0$. 

Notice that $u^\ep$ and $u_0$ can also be viewed as appropriate (unique) weak solutions of \eqref{pareps} and \eqref{parlim}. For example, cf. e.g. \cite{lions}, seek 
$u^\ep(x,t)\in L^2\left(0,T;\left(H_0^1(\Omega)\right)^n\right)\cap 
C\left([0,T]; \left(L^2(\Omega)\right)^n\right)$ with 
$\frac{\partial u}{\partial t}\in  L^2\left(0,T;H^{-1}(\Omega)\right)$, such that 
\beq
\langle\partial u^\ep/\partial t,\,v\rangle+\beta_\ep(u^\ep,v)=0
\label{weakpar}
\eeq
 for each 
$v\in \left(H_0^1(\Omega)\right)^n$ and a.e. $0\leq t\leq T$, where 
$
\beta_\ep(u,v):=\int_\Omega a^\ep(x)\nabla u\cdot \nabla v \,dx, 
$
with initial condition $u(x,0)=f^\ep(x)$. 
Then $u^\ep(x)=e^{-t A_\ep}f^\ep$ is readily seen to be the unique solution. 

This would allow a further refinement of Theorem \ref{parsem} to include (strong two-scale) convergence of 
the generalized fluxes 
$\xi^\ep(x,t):=\left(a^{(1)}(x/\ep)\right)^{1/2}\nabla u^\ep(x,t)$. 
This can be shown by first noticing that by setting in \eqref{weakpar} $u=v=u^\ep$ 
and recalling the $L^2$-boundedness of $f^\eps$ 
implies that $\xi^\ep(x)$ is uniformly 
bounded in $L^2(0,T;H_\ep)$, and hence (cf e.g. \cite{Past05} Lemma 4.4), up to a subsequence, 
$\xi^\ep(x,t) \stackrel{2}\rightharpoonup \xi_0(x,y,t)$ in $L^2(0,T;H)$. Selecting then in 
\eqref{weak} the test functions $v$ first as in the proof of 
Lemma \ref{lem22} we infer that $\xi_0\in L^2\left(0,T; L^2(\Omega; W)\right)$, and then as in the proof of Lemma \ref{fluxcorr} we conclude that $\xi_0(x,y,t)=Tu_0(x,y,t)$. Finally, we can show that in fact 
$\xi^\ep(x,t) \stackrel{2}\rightarrow \xi_0(x,y,t)$ following similar argument after Theorem 
\ref{s2src}. 

So, as Theorem \ref{parsem} implies, for the above generalization of a double porosity-type parabolic Cauchy problem
(cf e.g. \cite{Khrus, Zhikov2000}), 
the limit problem \eqref{parlim} can be derived under most general assumptions \eqref{rhobd}--\eqref{coerc1} and \eqref{keyass}. 

We emphasize that the condition that the two-scale limit of the Cauchy data $f_0(x,y)$ is in the subspace $H_0$ 
of $H=\left(L^2(\Omega\times Q)\right)^n$ but not in the whole of $H$ is important for Theorem \ref{parsem} to hold. If this 
condition is not met, the convergence to 
$u_0=e^{-tA_0}Pf_0$ (i.e. with $f_0$ replaced by its projection $Pf_0$ on $H_0$) 
would generally hold only in a weak sense and only `on the average' with 
respect to $t$, cf \cite{Past05} Theorem 1.6 and \cite{ZhP07} Theorem 2. 
\vspace{.15in}

Finally, following again \cite{Past05}, we provide a scenario ensuring 
 convergence of associated {\it hyperbolic} semigroups, with implications for two-scale homogenization of 
high-contrast hyperbolic problems. Consider the hyperbolic Cauchy problem
\beq
\rho^\ep(x)\frac{\partial^2 u^\ep}{\partial t^2}\,-\,\mbox{div}\left( a^\ep(x)\nabla u^\ep\right)\,=\,0, \ u^\ep(x,0)= f^\ep(x), 
\ \frac{\partial u^\ep}{\partial t}(x,0)= g^\ep(x),   
\label{hyperbeps}
\eeq
with initial data $f^\ep\in \left(H_0^1(\Omega)\right)^n$ and $
g^\ep\in \left(L^2(\Omega)\right)^n$. 
For every $\ep>0$ and $T>0$, the Cauchy problem \eqref{hyperbeps} is well-posed, cf e.g. \cite{lions}, for 
$u^\ep(x,t)\in C\left([0,T];\left(H_0^1(\Omega)\right)^n\right)$
with 
$\frac{\partial u}{\partial t}\in  C\left([0,T];\left(L^2(\Omega)\right)^n\right)$ and 
$\frac{\partial^2 u}{\partial t^2}\in  L^2\left(0,T;\left(H^{-1}(\Omega)\right)^n\right)$, such that 
\beq
\langle\partial^2 u^\ep/\partial t^2,\,v\rangle+\beta_\ep(u^\ep,v)=0
\label{weakpar3}
\eeq
 for each 
$v\in \left(H_0^1(\Omega)\right)^n$ and a.e. $0\leq t\leq T$. 
It is then routinely checked, referring to the spectral representation \eqref{spectheor} for  $A_\ep$,  that the unique solution of \eqref{hyperbeps} is 
\beq
u_\ep(x,t)\,=\,\cos\left(A_\ep^{1/2}t\right)f^\ep\,+\,
\frac{\sin\left(A_\ep^{1/2}t\right)}{A_\ep^{1/2}}g^\ep. 
\label{uepspectr}
\eeq
The Cauchy problem \eqref{hyperbeps}
can be interpreted in terms of a 
contraction semigroup on $\left(H^1_0(\Omega)\right)^n\times \left(L^2(\Omega)\right)^n$, cf. e.g. 
\cite{evans} \S 7.4.3 b. 
 Then, according to Theorem 5.2 of \cite{Past05}, a version of the Trotter-Kato theorem holds 
ensuring a weak two-scale convergence of related hyperbolic semigroups. Complementary or alternatively, 
one could exploit the self-adjointness and the non-negativeness of $A_\ep$ and $A_0$, cf \eqref{hyperbeps} and 
\eqref{uepspectr} and reduce the 
problem to that of a `Stone's unitary group'.  

Adapting e.g. Theorem 5.3 of \cite{Past05} to our case, we state the following theorem.  
\begin{theorem}\label{hyperbcauchy}
Let  
 $f^{\ep}\stackrel{2}\rightharpoonup f_0(x,y)\,\in U $, 
 $g^{\ep}\stackrel{2}\rightharpoonup g_0(x,y)\in H $, 
and let 
\beq
\limsup_{\varepsilon \to 0} \int_\Omega 
\,a^\ep (x)\nabla f^\ep\cdot \nabla f^\ep(x)\,
dx \,<\,\infty.  
\label{energbounded}
\eeq
Then 
for each $T>0$, for the solution $u^\ep(x,t)$ to the Cauchy problem \eqref{hyperbeps}, 
$u^\ep(x,t) \stackrel{2}\rightharpoonup u_0(x,y,t)$, 
$\left(a^{(1)}(x/\ep)\right)^{1/2}\nabla u^\ep(x,t) \stackrel{2}\rightharpoonup T u_0(x,y,t)$, 
$\frac{\partial u^\ep}{\partial t}(x,t) \stackrel{2}\rightharpoonup 
\frac{\partial u_0}{\partial t}(x,y,t)$ in $L^2(0,T;H_\ep)$\footnote{  
According to Definition 4.3 of \cite{Past05}, for a bounded sequence $v_\ep\in L^2(0,T;H_\ep)$ 
we say that $v^\ep(x,t) \stackrel{2}\rightharpoonup v(x,y,t)\in L^2(0,T;H)$ if 
for any $z^\ep(x)\stackrel{2}\rightarrow z(x,y)$ and any $\varphi(t)\in L^2(0,T)$, 
\[
\int_0^T\left(v^\ep(x,t), z^\ep(x)\right)_{H_\ep}\varphi(t)dt\,\to\, 
\int_0^T\left(v(x,y,t), z(x,y,t)\right)_{H}\varphi(t)dt. 
\]
}
where $u_0$ is the unique solution of two-scale 
Cauchy problem in $H_0$: 
\beq
\frac{\partial^2 u_0}{\partial t^2}\,+\,A_0u_0\,=\,0, \ \ u_0(x,y,0)= f_0(x,y), \  \ 
\frac{\partial u_0}{\partial t}(x,y,0)= Pg_0(x,y). 
\label{hyperblim}
\eeq
\end{theorem}
The limit Cauchy problem \eqref{hyperblim} is well-posed, cf \cite{lions}, for 
$u_0(x,t)\in C\left([0,T];U\right)$ 
with 
$\frac{\partial u}{\partial t}\in  C\left([0,T];H_0\right)$ and 
$\frac{\partial^2 u}{\partial t^2}\in  L^2\left(0,T;U^*\right)$, (where $U^*$ denotes the dual space to $U$), such that 
\beq
\langle\partial^2 u_0/\partial t^2,\,v\rangle+\beta(u_0,v)=0
\label{weakpar2}
\eeq
 for each 
$v\in U$ and a.e. $0\leq t\leq T$. 

Then, according to the spectral representation \eqref{spectheor} for  $A_0$,  the unique solution of \eqref{hyperbeps} is readily seen to be 
\beq
u_0(x,t)\,=\,\cos\left(A_0^{1/2}t\right)f_0\,+\,
\frac{\sin\left(A_0^{1/2}t\right)}{A_0^{1/2}}Pg_0. 
\label{u0pectr}
\eeq

Note that, compared to the abstract Theorem 5.3 of \cite{Past05}, we have again stated here a further refinement 
specific to at least our general class of the problems: on the weak two-scale convergence of the generalized fluxes 
$\xi^\ep(x,t):=\left(a^{(1)}(x/\ep)\right)^{1/2}\nabla u^\ep(x,t)$, which is a natural generalization of the weak $H^1$-convergence of $u^\ep$ in the classical homogenization.   Indeed  
 by \eqref{energbounded}, the $L^2$-boundedness of $g^\eps$ and the energy conservation for \eqref{hyperbeps}, $\xi^\ep(x)$ is uniformly 
bounded in $L^2(0,T;H_\ep)$, and hence the convergence can be shown to hold via 
a straightforward modification of the proof of a similar convergence for parabolic problem as outlined below 
Theorem \ref{parsem}.


One can similarly adopt Theorem 7.2 of \cite{Past05} to establish a sufficient condition 
on the initial data $f^\ep(x)$ and $g^\ep(x)$ 
for appropriate strong 
(pointwise in $t$) convergence of $u^\ep(x,t)$ to $u_0(x,t)$. 

\vspace{.1in} 

The above implications may be interpreted as follows. Under generic assumptions on the degeneracy 
$a^{(1)}(y)$, notably under the key decomposition assumption \eqref{keyass} together with the original 
assumptions \eqref{rhobd}--\eqref{coerc1}, for a wide class of domains $\Omega$ (Remark \ref{remepigr}) the limit resolvent problem as well as the limit 
parabolic and hyperbolic Cauchy problems retain the two-scale pattern of respectively the right hand side 
and of the Cauchy data. That is in contrast with the spectral problem (see the discussion in following 
Corollary \ref{h1spec} above), 
which may generally retain a quasi-periodic pattern in the limit and which may hence need to be reflected 
by appropriately extending the limit operator, unless some additional conditions are imposed. 
The latter may deserve a separate investigation, as well as whether the cases where \eqref{keyass} is not satisfied e.g. the examples of high anisotropy in \cite{CSZ06, VPS09} can also be treated generally, possibly by 
combining the presented ideas with those based on convergence with respect to measures \cite{Zhikov2000}. 
The latter approach has indeed proved working in \cite{CSZ06} where \eqref{keyass} is not satisfied. 
It may also be of interest to investigate general properties of the limit operator $A_0$ and of associated 
two-scale coupled limit problems, and in particular under what conditions the scales could be uncoupled, in one 
or another way.

\end{document}